\documentclass[poms,final,nonblindrev]{poms1_V1} 

\OneAndAHalfSpacedXI 

\usepackage{graphicx}
\usepackage{subfigure, epsfig}
\usepackage{natbib}

\usepackage{multirow,xcolor,latexsym,amssymb,amsmath,epsfig,amsfonts,algorithm}
\usepackage[title]{appendix}
\usepackage{bbm,bm}
\usepackage{graphicx}
\usepackage[margin=1in]{geometry}
\usepackage{makecell}
\usepackage{multicol}

 \bibpunct[, ]{(}{)}{,}{a}{}{,}%
 \def\bibfont{\small}%
\TheoremsNumberedThrough     
\ECRepeatTheorems
\newtheorem{observation}{Observation}
\EquationsNumberedThrough    

\begin{document}


\RUNAUTHOR{Hassin and Wang}

\RUNTITLE{Self Selection of Service Type}

\TITLE{Optimal and Self Selection of Service Type in a Queueing System where Long Service Postpones the Need for the Next Service}

\ARTICLEAUTHORS{%
\AUTHOR{Refael Hassin}
\AFF{School of Mathematical Sciences, Tel Aviv University, \EMAIL{hassin@tauex.tau.ac.il}} 
\AUTHOR{Jiesen Wang}
\AFF{Korteweg-de Vries Institute for Mathematics, University of Amsterdam \EMAIL{jiesenwang@gmail.com}}
} 

\ABSTRACT{%
We study a make-to-order system with a finite set of customers.  Production is stochastic with a nonlinear dependence between the ordered quantity and the production rate.  Customers may have to queue until their turn arrives, and therefore their order decisions interact.  Specifically, while being served, customers are aware of the queue length and choose one of two order quantities (or service types). The time to the next replenishment (their activity time) is stochastic and depends on the order quantities. A customer is inactive during service and while waiting in the queue. We refer to the type of service with a greater ratio of expected activity to service time as ``more efficient''. In the centralized case, the system is interested in maximizing the steady-state average number of active customers, which is referred to as the efficiency of the system.  We show that choosing the more efficient service is not always optimal, but the optimal strategy can be approximated well by selecting one of three threshold strategies which depend on the number of inactive customers.  In the decentralized case, each customer acts to maximize the fraction of time she is active. We observe that individuals and the manager have opposite incentives: When the queue is long, individuals tend to choose the long service, while the manager prefers the short service in this case.  This makes the system difficult to regulate. However, we show that simply removing the less efficient service significantly increases efficiency.
}%

\HISTORY{This paper was first submitted on April 20, 2023, has been with the authors for 18 months for 2 revisions, and was accepted on July 19, 2025. \\
Handling Editor: Michael Pinedo\\
Corresponding author: Jiesen Wang}

\maketitle

%


\section{Introduction}

In a common situation, periodic decisions must be made that affect both the immediate reward (or cost) and the time of the next decision.  We mention just a few cases, but clearly, there are innumerable similar situations. In a typical make-to-order inventory planning situation, a firm or an individual decides on the order quantity of some good, and the higher the ordered quantity, the longer the time until there will be a need for the next order.
In a different situation, the more one invests in preventive maintenance, the longer the time until the next maintenance.
Maintenance can be classified according to the degree to which the system is restored. Perfect maintenance restores the system to be as good as new. Minimal maintenance restores the system to the state whose failure rate is the same as it was when it failed.  See  \citet{PW96} and \citet{DS20} for comprehensive reviews of this line of research. 

Common to these cases is the choice between frequent cheap expenditures and less frequent but more expensive expenditures.  The utility per decision cycle is $f(T) - K$,
where $K >0$ is a fixed cost, $T$ is the cycle length, and $f(T)$ is concave increasing variable reward. The decision problem is
\begin{equation} \label{eq:problem}
    \max_{T > 0} \, \frac{f(T)}{T} - \frac{K}{T} \,.
\end{equation}

We study a decision model of the above kind, but add to it interactions among the decision makers which emerge from the need to share a common queue when waiting for their order to be fulfilled. These interactions associate external effects with individual decisions, and in particular, their individual selfish strategies may differ from the overall (social/system) efficient strategy. Moreover, unlike the problem in \eqref{eq:problem}, decisions may be state-dependent.

In an illustrative example, electric vehicle (EV) drivers need to charge the car for some amount of time before driving, and then return to the charging station to recharge when it is running out of power. Since charging time is not negligible, queues at charging stations are inevitable. According to China Daily news\footnote{https://global.chinadaily.com.cn/a/202110/11/WS61639f31a310cdd39bc6e0b7.html}, an EV driver had to spend four hours at the rest area in Hunan province, waiting for her turn to charge her vehicle for the one hour it needed. The battery recharge curve\footnote{https://www.vintagetrailersupply.com/progressive-dynamics-pd9245-45-amp-converter-vts-320/} shows that the time to reach full charge is almost three times, and can be five times the time to reach 90\% of the maximum charge, see Figure 10 of \citet{KSK20}. So, the relationship between the gained power and the charging time is in general concave. 

In this paper, we analyze a stylized model with two types of service and a finite number of customers. (The low- and high-rate service options in our model may correspond to perfect and minimal options in the maintenance application.) While queueing or obtaining service, the customers of the system are inactive, but once their service terminates,, they become active until they need the next service. Fast service generates shorter active time. For example, the fast and slow services may correspond to 80\% and 100\% charging percent in the battery recharge curve. We are interested in the state-dependent choice of service type that maximizes the {\it efficiency of the system} defined as the average fraction of active time per unit time, times the number of customers, or, equivalently, the average number of active customers per unit time.  We also analyze the equilibrium selection strategy when customers act to maximize their individual welfare. We observe the `follow the crowd' phenomenon, which generally leads to multiple Nash equilibria.  In our model, the inactivity time during service and queueing substitutes the fixed cost $K$ in (\ref{eq:problem}).  Therefore, this ``ordering cost’’ is now {\it endogenous} and affected by the service type selection strategy.

Let the {\it efficiency of a service type} be the ratio of expected activity time to expected service duration per service. A service-type selection strategy is (socially) optimal if it maximizes system efficiency. In the case of a single customer, always choosing the more efficient service type is optimal, but this is not the case with multiple customers, whose choices depend on the state of the system.
We support the conjecture that if we constrain the strategy to be a threshold strategy depending on the number of inactive customers, but not the type of service the active customers receive, then the resulting best strategy is near optimal. Therefore, we focus on threshold strategies that depend on the number of inactive customers only.  

We note an interesting phenomenon - although long service is the (socially) preferred choice when the queue is short, individuals may prefer long service when the queue is long, if all the other individuals choose long service when the queue length exceeds a threshold. If this is the case, traditional ways to induce customers to behave in the system's optimal way cannot be used. However, we numerically show that simply removing the less efficient service type is an effective regulation.  

\subsection{Literature review}

{Our model is inspired by situations in which customers join a queue for replenishing their inventory of some good.  They decide on their order size - a large order will lengthen both the production time and the time until the next replenishment (the activity time).  However, the relation between production time and activity time may be complex.  In some cases, production time is a convex function of the order size, and in some cases the activity time is a concave function of the order size.  An important case is that of perishable products with a random lifetime.

There is an extensive literature on queueing-inventory systems, see for example, the survey of \cite{KS16}.  While these systems also combine queues and replenishment, their motivation is very different, and it is assumed that
the server maintains an inventory of some raw material required to perform the service.  This is in contrast with our model, where the inventory belongs to the customers.  Our model, when viewed as an inventory model, belongs to the growing literature of inventory systems with strategic customers \citep[see][]{GH86,LS14,HLS19}. \citet{LS14} consider the inventory control for a system in which the production cost consists of a fixed cost and a piecewise linear convex cost. \citet{HLS19} study the joint pricing and inventory control problem for a single product, where the production cost includes a fixed cost and a convex or concave variable cost. Both of the above models are about the problem in \eqref{eq:problem}. 


Game theoretic models of EV charging are described by \citet{DMP18} and \citet{DHJBW21}. \cite{SEK14} analyze various charging strategies for electric transport vehicles, where the economic assessment is based on a field experiment in a closed system. \cite{RRKJMH15} study the allocation of multiple mobile robots to docking stations, based on the robot's task priority, location awareness, and the docking station. \cite{GB19} consider the problem of assigning a fixed number of mobile robots to charging stations, to minimize the total charging time. Similar problems of a fixed number of mobile robots are also analyzed by \cite{CV09}, \citet{ZXK18}, and \citet{CFWX21}.
{\citet{WLLZMG22} consider a closed network of rechargeable sensors and a mobile charger that selects optimal charging positions and power levels from a discrete set of possibilities.  

Some open network models assume that customers enjoy longer service, as in our case.  See, for example,  \cite{H98}, \cite{H01}, \cite{YHI07}, \cite{TR14}, \citet{J22}, \citet{FS22}, and Section 6.3.2 in \citet{H16} on expert systems.  Similar to our model, \cite{CV83} consider a closed system with two players sharing an FCFS server and maximizing
their activity time (corresponding to the time they spend in service). The work in \cite{CV83} is motivated by programs that require access to a single shared disk (resource) between computations. As in our case, the active and inactive times are exponentially distributed, but in contrast to our assumption,  they are decided by the customers subject to an exogenous constraint on their ratio. 
\citet{CZ16} consider holiday shopping events and mention that long queues induce customers to purchase larger quantities. In our case, under the equilibrium that results from selfish optimization, long queues induce the choice of slow service.

A similar model to ours is first studied in \citet{BH19}, who consider a system of two customers and assume that the state of the system is {\it unobservable} to the customers. The authors derive the conditions for the existence of equilibria. They also show that a pure asymmetric equilibrium does not exist, a pure asymmetric strategy cannot be optimal, and a pure symmetric equilibrium always exists.

Another similar decision problem concerns advertising effectiveness and optimal frequency \citep[see][]{B00,GH10}. Shall we conduct expensive advertising campaigns at a low frequency or cheap ones at a higher frequency? In these models, firms are decision-makers.

\subsection{Our contribution}


We consider a closed Markovian queueing system with $N$ customers, and study the strategy that maximizes system efficiency when fast and slow services are available. A major challenge of this research is that customer decisions interact, and the flow inside the network, which is affected by the service type adopted by the customers, is not Poisson.
Intuitively, it is optimal to use the more efficient service type, but it is not that simple because the waiting time in the queue can make the more efficient service type less desirable.
We can view the server as producing active time. With the more efficient service type, more activity time is produced per unit service.
However, the server is not constantly busy and to achieve maximum efficiency one must also take into consideration the probability that the server is busy, which may be greater under less efficient service.

Interestingly, when $N = 1$ and $N = \infty$, it is optimal to always use the more efficient service, but when $N$ takes intermediate values, this is not the case.


 When $N = 2$, we obtain the system efficiency expressions for all pure strategies and prove that the optimal strategy is one of the three: always choose the low rate, always choose the high rate, or choose the high rate when the other customer is queueing and choose the low rate otherwise. In particular, the third strategy is optimal when the two types of services are equally efficient. This strategy means that the manager wants to dispel the queue quicker when there is a queue.

When $N = 3, 4$, we apply brute-force calculation to numerically obtain the optimal strategy. Specifically, for a fixed parameter setting, 
we compute the efficiency for all the (64 for $N = 3$ and 1024 for $N=4$) strategies, and then note down the strategies that result in maximum efficiency. Interestingly, for some parameter settings, the optimal strategy can choose the low rate when the queue is long but the high rate when the queue is shorter, which is different from the case where $N = 2$. This is because in this closed network, we need to consider not only the immediate reward brought by the decision, but also the expected reward of the next state in which this decision results. We will explain this in detail, using a dynamic programming approach, in Section \ref{sec:DP}. 


We observe that the choice of service type under the optimal strategy depends not only on the number of inactive customers but also on the number of active customers with different rates. However, if we constrain the strategy to only be based on the number of inactive customers, the resulting optimal system efficiency does not differ much from the optimal system efficiency. In other words, the information on the types of active customers is not significant. Thus, for $N > 4$, we focus on threshold strategies that depend only on the number of inactive customers and develop an efficient way to calculate the system efficiency for a given threshold strategy. We also find that when $N \rightarrow \infty$, the optimal strategy is to implement the more efficient service mode in any state.

Finally, we show that the best among the three strategies (a) always choose long service, (b) always choose short service, and (c) choose long service if there is only one inactive customer, gives a near-optimal solution. This is numerically demonstrated for the case where $N = 10$.}

Another aspect we consider in our paper is the decentralized case. Suppose that each customer optimizes her expected active time per unit time and is aware that other customers do the same. We assume bounded rationality here \citep[see, for example,][Chapter 11]{H16}, that is, queueing customers cannot observe or remember the service types chosen by customers in front of them, and they follow simple decision rules. Specifically, we assume that the customers' decision is based only on the number of inactive customers. The current queue length when a customer leaves is correlated with the queue length when she returns. Intuitively, the longer the queue when the customer is served, the more reluctant she is to come back soon to the queue. Thus, we assume that customers choose the low rate if and only if the queue length is above a certain threshold.} This means that the type of threshold strategy here is ``opposite to '' that of the centralized optimal case. We find that the best-response threshold of the customer in service is higher when the threshold chosen by other customers is higher. This phenomenon is termed as ``follow the crowd" in \citet{HH97} and \citet{HH03}, and multiple equilibria are possible in this case.

As we just noted, individuals and the manager may have opposite incentive directions.  This means that there are no simple ways, such as charging type-dependent service fees, to induce a desired social equilibrium. 
{We examine the effect of providing only the more efficient service type, and demonstrate that it induces satisfactory results.}  We also calculate the ``price of anarchy'' (see \cite{GH21}, for a survey on the price of anarchy in stochastic queueing systems) and show that it can be arbitrarily large. However, we show that by simply removing the less efficient service, the system efficiency improves a lot and the resulting individual behavior is close to optimal.

The remainder of this paper is organized as follows. Section \ref{sec:model} briefly describes the model and its relevant parameters. In Section \ref{sec:optimization}, we perform an optimization analysis of efficiency by deriving the analytical solution for $N = 2$ and numerically calculating the efficiency for every possible pure strategy for $N = 3,4$. We propose a method to calculate the efficiency of the system for any threshold strategy for a general value of $N$. In Section \ref{sec:NE}, we study efficiency optimization at an individual level and the Nash equilibrium threshold. Our numerical analysis shows that there always exists at least one Nash equilibrium of this type. In Section \ref{sec:FR}, we discuss possible extensions.

\section{The model} \label{sec:model}

We consider a closed queueing network with a single server, $N$ customers, and two types of service. When customer service ends, the customer stays active for an exponentially distributed time and then returns for the next service. If the server is occupied, inactive customers queue for service. The service discipline is first-come first-served. The service time is exponentially distributed and there are two types of service: {high} and {low}.

If a customer chooses the high rate, then her expected service time will be shorter, but her expected active time will also be shorter. Let $\mu_h$ and $\lambda_h$, $\mu_l$ and $\lambda_l$ denote the service and activity rates of the fast and slow service, respectively. Then $\mu_h > \mu_l$ and $\lambda_h > \lambda_l$.

We first aim to maximize system efficiency, which is the average number of active customers. To calculate efficiency, we construct a continuous-time Markov process in the state space,
\[
\mathcal{S} \equiv \{(i,h), i = 0, 1, \ldots, N, \, h = 0,1,\ldots, N-i\} \,,
\]
where $i$ and $h$ denote the number of inactive customers and the number of customers active with high rate, respectively. Note that the other $N-i-h$ customers are active at a low rate.

\section{System optimization} \label{sec:optimization}

A strategy prescribes for customers in service (when the server is busy) a service type for each state in $\mathcal{S}$ with $i\ge1$, note that $(0,h)$ does not require action. Thus, a strategy is a function $a(i,h)\in[0,1]$, which means that when the state is $(i,h)$, the customer chooses a high-rate service with probability $a(i,h)$. Under strategy $a$, a customer can change her service type when the state changes. This does not cause extra difficulty because of the memoryless property of the exponential distribution. 
We denote a specific strategy by aligning $a(i,h), ~i > 0,~ 0 \leq h \leq N-i$ in lexicographic order:
\[
 a(1,0)a(1,1)\ldots a(1,N-1)|a(2,0)a(2,1)\ldots a(2,N-2)|\ldots |a(N,0) \,.
\]

We define the system output as the average number of active customers, or equivalently, the steady-state probability to be active times $N$.
When the state is $(i,h), i < N$, there are $N-i$ customers active. Thus, the efficiency of the system under strategy $a$ is
\begin{equation} \label{eq:SW}
    S(a) \equiv \sum_{i = 0}^{N-1}\sum_{h=0}^{N-i} \,  (N-i) \, \pi^{(a)}_{i,h} \,,
\end{equation}
where $\pi^{(a)}_{i,h}, \, (i,h) \in \mathcal{S}$ is the steady-state probability of state $(i,h)$ under strategy $a$.

We can also consider the system efficiency from the server's perspective. The server contributes to efficiency only when it is busy, so we only consider states $(i,h)$ with $i > 0$. When a server is working at a high rate, then after being busy for a time that is exponentially distributed with the rate $\mu_h$, there will be a customer being active for a time that is exponentially distributed with rate $\lambda_h$. That is, on average, each unit of service time contributes $(1/\lambda_h)/(1/\mu_h)$ active time. Similarly, when a server is working at a low rate, on average, each unit of service time contributes $(1/\lambda_l)/(1/\mu_l)$ active time. Thus, the system efficiency can also be written as
\begin{equation} \label{eq:SW2}
    S(a) =  \sum_{i = 1}^{N} \sum_{h=0}^{N-i} \, \left(a(i,h) \frac{\mu_h}{\lambda_h} + (1-a(i,h)) \frac{\mu_l}{\lambda_l}\right) \, \pi^{(a)}_{i,h}  \,.
\end{equation}

When $N = 1$, there is only one customer in the system.  If this customer uses the high rate or low rate service, the efficiency of the system for that type will be $(1/\lambda
_h)/(1/\lambda_h+1/\mu_h)$ or $(1/\lambda
_l)/(1/\lambda_l+1/\mu_l)$, respectively. Thus, the manager wants to stick to the more efficient service.

In this section, we first analyze the (socially) optimal strategy for $N = 2, 3, 4$, then move to the case of a general $N$. When $N = 2,3,4$, we consider strategies with complete information, that is, when making decisions, the system manager is aware of the number of active customers at a high rate, at a low rate, and inactive customers. In the full information case, the optimal strategy is w.l.o.g. {\it pure}. Later in Section \ref{sec:NE}, {\it mixed} strategies may be necessary when we deal with individual optimization and equilibrium.

Let $\mathcal{T}$ be the set of threshold policies depending on the number of inactive customers. When $N = 2$, the optimal strategy is of threshold type. When $N = 3,4$, we experimented with a wide range of parameter values and observed that although the optimal strategy may not be of threshold type, the relative loss by using a threshold strategy is quite insignificant. That means that, although the optimal strategy may not be of threshold type, the optimal system efficiency does not differ significantly from the best efficiency if we constrain the strategies to the threshold type. Thus, when $N > 4$ we ignore the value of $h$, and focus only on (threshold) strategies that are a function of $i$, the number of inactive customers.

\subsection{The case $N = 2$}


\begin{figure}[h]
    \centering
    \begin{minipage}{0.45\textwidth}
        \centering
        \includegraphics[width=\linewidth]{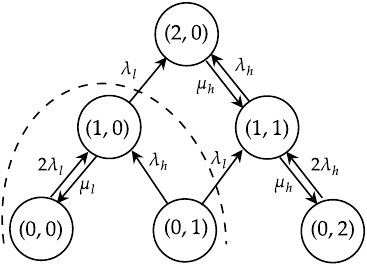}
        \textbf{(a)} 
        The transition rate diagram for all the states, under strategy $01|1$. The dashed curve separates the transient states from the recurrent states.
    \end{minipage}%
    \begin{minipage}{0.42\textwidth}
        \centering
        \includegraphics[width=\linewidth]{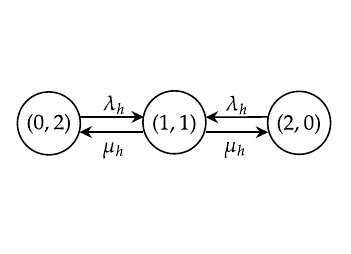}
        \textbf{(b)}
        The transition rate diagram for recurrent states, under strategies $01|1$ and $11|1$.
    \end{minipage}
    \vspace*{10pt}
    \caption{Transition rate diagrams for different state sets} 
    \label{fig:TRD2}
\end{figure}

When there are two customers in the system, there are three possible states $(i,h)$ a customer in service may face: $(1,0), (1,1), (2,0)$. Denote the strategy by 
\[
a(1,0)\, a(1,1)\,| a(2,0) \,.
\]
For each strategy, there are specific transition rates for each $(i,h) \in \mathcal{S}$, based on which we can calculate the steady-state probabilities. For example, if the strategy is $01|1$, the transition rate diagram is depicted in Figure \ref{fig:TRD2}(a). Note that in this case, three states $(0,0),(0,1),(1,0)$ are transient. States $(0,0)$ and $(0,1)$ do not require action as no one is at service, the decision at state $(1,0)$ does not affect system efficiency. If we only keep the recurrent states, the transition rate diagram is depicted in Figure \ref{fig:TRD2}(b), and the efficiency calculated from it is the same as for the strategy $01|1$ or $11|1$. We will use `$*$' in the strategy to denote the action for the transient states. For example, we use ${*}1|1$ to denote strategy $01|1$ or $11|1$ because they result in the same efficiency. Following similar reasoning, we denote $01|0$ and $00|0$ as $0{*}|0$. In the following discussion, we only consider strategies $0{*}|0, {*}1|1, 00|1, 11|0, 10|0$, and $10|1$. The steady-state probabilities under the six strategies are calculated and shown in Appendix \ref{appendix:1}, with which we can calculate the efficiency of the system using \eqref{eq:SW} or \eqref{eq:SW2}. 


The first interesting result we have is for strategies $10|0$ and $10|1$. It is summarized in Lemma \ref{lemma:101}, whose proof is in Appendix \ref{appendix:101} and \ref{appendix:100}.
\begin{lemma} \label{lemma:101}
	The strategy $10|0$ or $10|1$ is never optimal.
\end{lemma}
\noindent Intuitively, the strategies $10|0$ and $10|1$ cannot be optimal. In state $(1,0)$, the other customer is active at a low rate, while in state $(1,1)$, the other customer is active at a high rate. So the active customer will come to the server sooner if it is $(1,1)$. If the manager selects the high rate when it is $(1,0)$, then she has a higher incentive to choose a high rate at $(1,1)$.

The next interesting result we have is for the strategy $11|0$. It is presented in Lemma \ref{lemma:110}, whose proof is in Appendix \ref{appendix:110}.
\begin{lemma}\label{lemma:110}
	The strategytegy $11|0$ is never optimal.
\end{lemma}
\noindent Strategy $11|0$ means that the manager will choose a high rate if there is an active customer and choose a low rate if both customers are queueing. This is intuitively not optimal as the manager has a higher incentive for dispelling the queue when both customers are queueing.

From Lemmas \ref{lemma:101} and \ref{lemma:110} it follows that the optimal strategy when $N = 2$ can only be ${*}1|1,00|1$ or $0{*}|0$. The following theorem details the parameter regions where ${*}1|1,00|1$ or $0{*}|0$ is optimal. The proof of Theorem \ref{thm:n2} is in Appendix \ref{appendix:thmoptimal}.
\begin{theorem} \label{thm:n2}
	For any fixed $\lambda_l$ and $\mu_l$, there exist two switching curves $\lambda_h = f(\mu_h)$ and $\lambda_h = g(\mu_h)$ such that
	$f(\mu_h) \geq g(\mu_h), \forall \mu_h$.
	If $\lambda_h > f(\mu_h)$, $0{*}|0$ is optimal; if $g(\mu_h) <\lambda_h < f(\mu_h)$, $00|1$ is optimal; if $\lambda_h < g(\mu_h)$, ${*}1|1$ is optimal. In particular, when $\lambda_h/\mu_h = \lambda_l/\mu_l$, $00|1$ is optimal.
\end{theorem}

The optimal strategies for different values of $\lambda_h$ and $\mu_h$ when $\mu_l = 1$ and $\lambda_l = 0.5, 0.7, 0.9$ are depicted in Figure \ref{fig:Optimalan2}. It is optimal to always stick to a service type if it is much more efficient. If the two service modes do not differ very much in terms of service efficiency, adopting the fast service when the other customer is queueing and the slow service otherwise is optimal. Fix $\lambda_l, \, \mu_h,$ and $\mu_l$, then as $\lambda_h$ increases, the service efficiency of the fast service decreases, and the low service becomes more efficient compared to the fast service, so the optimal strategy changes from $*1|1$, to $00|1$, then to $0{*}|0$.


\begin{figure}[!ht]
    \centering
    \begin{minipage}{0.33\textwidth}
        \centering
        \includegraphics[width=\linewidth]{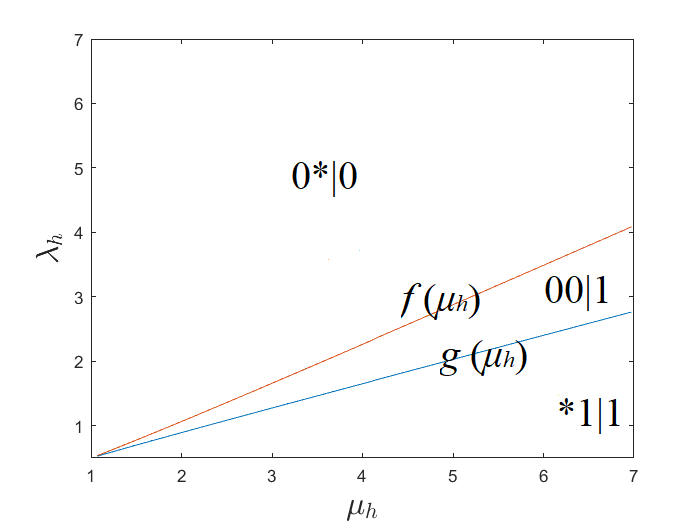}
        \textbf{(a)} 
        $\mu_l = 1, \lambda_l = 0.5$
    \end{minipage}%
    \begin{minipage}{0.33\textwidth}
        \centering
        \includegraphics[width=\linewidth]{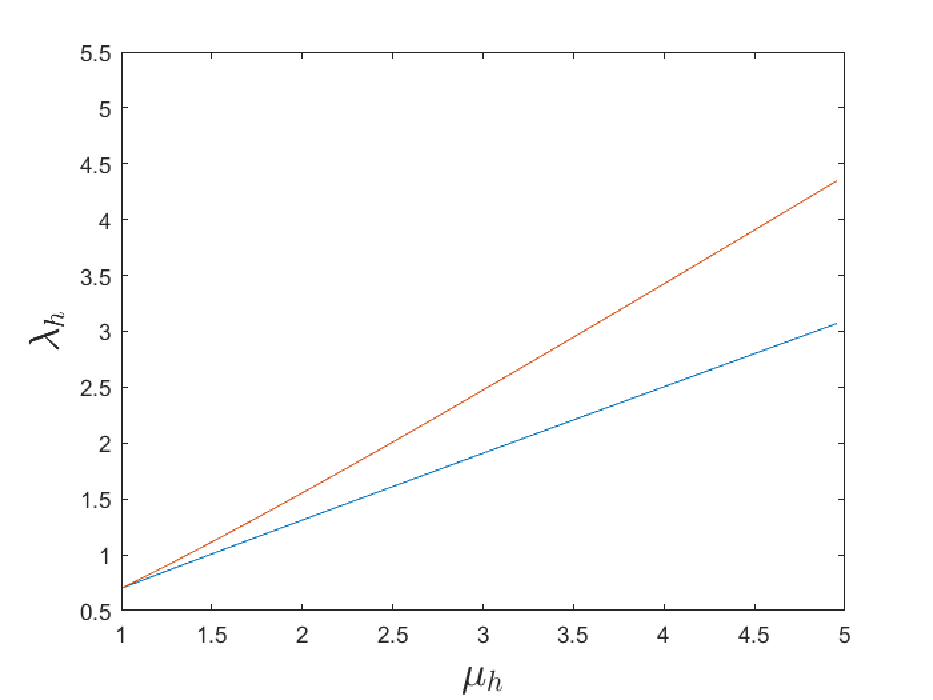}
        \textbf{(b)}
        $\mu_l = 1, \lambda_l = 0.7$
    \end{minipage}%
    \begin{minipage}{0.33\textwidth}
        \centering
        \includegraphics[width=\linewidth]{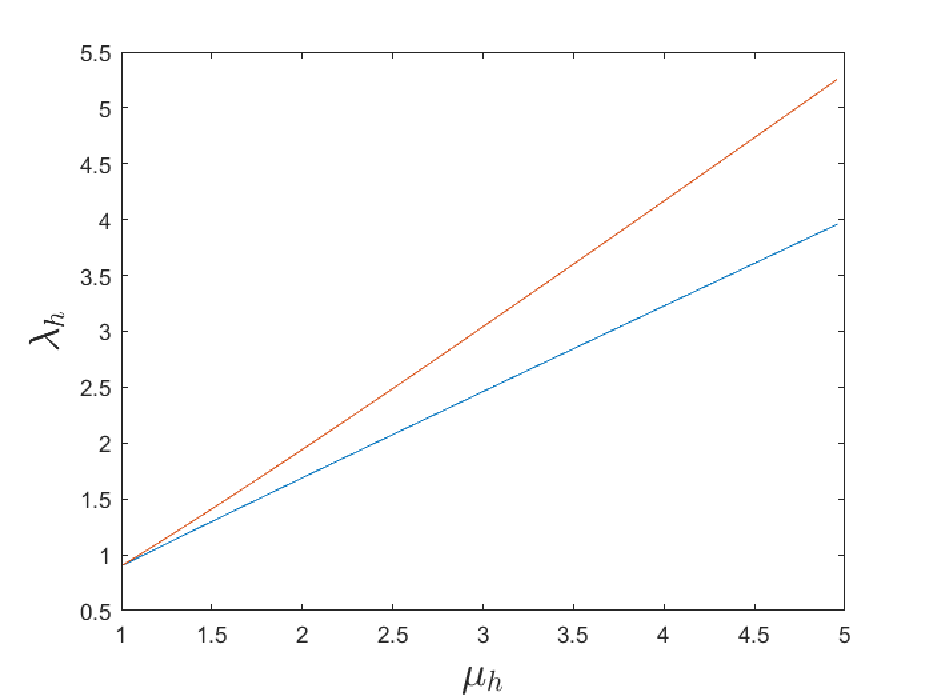}
        \textbf{(c)} 
        $\mu_l = 1, \lambda_l = 0.9$
    \end{minipage}
    \vspace*{10pt}
    \caption{The optimal strategy regions for different parameter settings when $N = 2$.} \label{fig:Optimalan2}
\end{figure}

\subsection{The case $N = 3$}

When $N = 3$, we denote the strategy by
\[
a(1,0)a(1,1)a(1,2)|a(2,0)a(2,1)|a(3,0) \,.
\]
There are two action options (high or low rate) for each $a(i,h)$, therefore, there is a total of 64 possible strategies. Similar to strategies $01|0$ and $00|0$, or $01|1$ and $11|1$, in the $N=2$ case, different strategies also may lead to the same stationary distribution, thus the same system efficiency when $N = 3$. For example, if we use strategy $000|10|0$, the transition rate diagram is shown in Figure \ref{fig:TRD3}, from which we can see that the state $(1,2)$ is transient. (States $(0,2), (0,3)$ are also transient, but do not require any action.) Thus the action on state $(1,2)$ does not matter. In this case, we use $00{*}|10|0$ to denote the strategy. 
\begin{figure}[h]
	\centering
	{\includegraphics[width=0.5\linewidth]{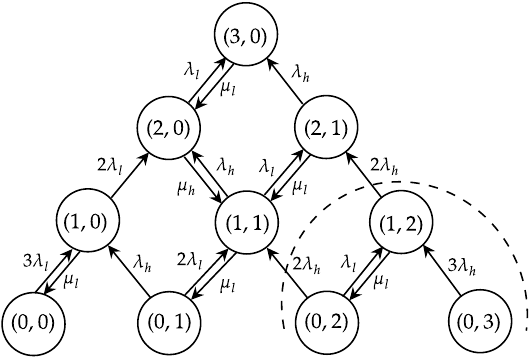}}
	\caption{The transition rate diagram for strategy $000|10|0$.  The dashed curve separates the transient states from the recurrent states.} \label{fig:TRD3}
\end{figure}

For each strategy, we can write down the transition rate diagram similar to Figure \ref{fig:TRD3}, based on which we calculate the optimal strategy. The optimal strategy among the $64$ strategies for different parameter settings when $\mu_l = 1$ and $\lambda_l = 0.5, 0.7, 0.9$ are represented in Figure \ref{fig:Optimalan3}.  Based on our numerical results we reach the following conclusion.
\begin{observation} \label{conjecture:1}
	For any fixed value $\lambda_l$ and $\mu_l$, there exists a value $\hat{\lambda}_h$ such that if we fix $\lambda_h$, when $\mu_h$ increases from $\mu_l$ to $\infty$, the evolvement of the best strategy is
	\[
	0{**}|0{*}|0 \rightarrow 00{*}|00|1 \rightarrow 00{*}|10|1 \rightarrow 000|11|1 \rightarrow {**}1|{*}1|1
	\]
	if $\lambda_l < \lambda_h \leq \hat{\lambda}_h$, and
	\[
	0{**}|0{*}|0 \rightarrow 00{*}|10|0 \rightarrow 00{*}|10|1 \rightarrow 000|11|1  \rightarrow {**}1|{*}1|1
	\]
	if $\lambda_h > \hat{\lambda}_h$.
\end{observation}


\begin{figure}[ht]
    \centering
    \begin{minipage}{0.33\textwidth}
        \centering
        \includegraphics[width=\linewidth]{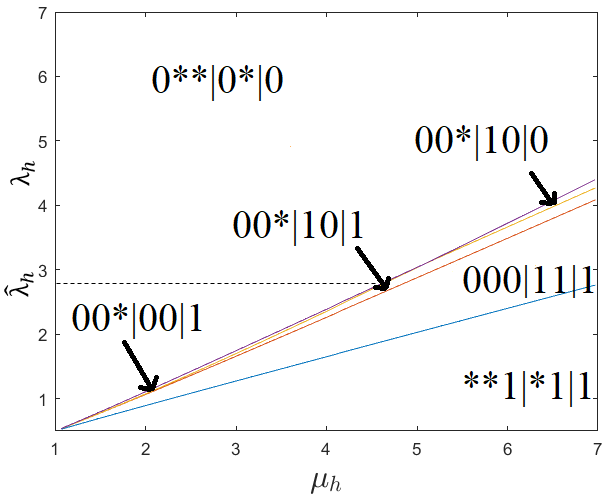}
        \textbf{(a)} 
        $\mu_l = 1, \lambda_l = 0.5$
    \end{minipage}%
    \begin{minipage}{0.33\textwidth}
        \centering
        \includegraphics[width=\linewidth]{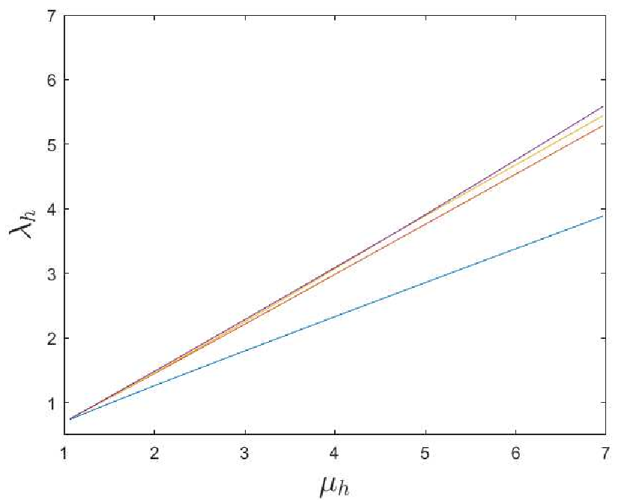}
        \textbf{(b)} 
        $\mu_l = 1, \lambda_l = 0.7$
    \end{minipage}%
    \begin{minipage}{0.33\textwidth}
        \centering
        \includegraphics[width=\linewidth]{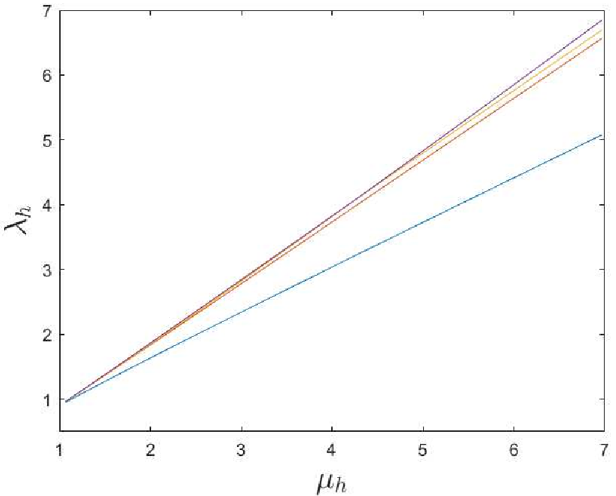}
        \textbf{(c)} 
        $\mu_l = 1, \lambda_l = 0.9$
    \end{minipage}
    \vspace*{10pt}
    \caption{The optimal strategy regions for different parameter settings when $N = 3$.}
    \label{fig:Optimalan3}
\end{figure}
	
Here, we do not look into the value of $\hat{\lambda}_h$ and only explain the evolvement of strategies with $\mu_h$ if we fix $\lambda_l$, $\mu_l$, and $\lambda_h$. When $\lambda_h \leq \hat{\lambda}_h$, if $\mu_h$ is small, the optimal strategy is to use the low rate service in any state, since the service efficiency $\mu_h/\lambda_h$ is low. As $\mu_h$ increases, the optimal strategy switches to using the high rate if and only if the two other customers are inactive. As $\mu_h$ keeps increasing, it is optimal to use the high rate if the two others are either inactive or one is inactive and the other is active at a low rate. As $\mu_h$ keeps increasing, it is optimal to use the high-rate service when there is at least one other inactive customer. When $\mu_h$ is large enough, it is optimal to use the high-rate service in any state.
	
When $\lambda_h > \hat{\lambda}_h$, the optimal strategy is still using the low rate in any state when $\mu_h$ is small and as $\mu_h$ increases, using the high rate when there is at least one other inactive customer. When $\mu_h$ is large enough, it is optimal to use the high-rate service in any state. However, strategy $00{*}|10|0$ is possible. This strategy is to use the low rate when the two other customers are inactive, and use the high rate when one other is inactive and the other is active at a low rate.
	
Figure \ref{fig:Optimalan3} shows both evolvement cases for three parameter settings. Note that the line $\lambda_h = \mu_h \, \lambda_l/\mu_l$ is between the line separating $00{*}|10|1$ and $000|11|1$, and the line separating $000|11|1$ and ${**}1|{*}1|1$. That is, when $\lambda_h = \mu_h \, \lambda_l/\mu_l$, $000|11|1$ is optimal. We proved in Lemma \ref{lemma:110} that when $N=2$, strategy $11|0$ is never optimal, because the queue motivates high service. That is, the manager wants the customer in service to choose a high rate when the queue is long and a low rate when the queue is short. Our findings lead to Conjecture \ref{Conj:1}, that when the two services are equally efficient, choosing the low rate if and only if there is only one inactive customer is optimal for any $N$. Interestingly, when $N = 3$, strategy $00{*}|10|0$ may be optimal. This means adopting {\it low} rate when the other two customers are inactive, and {\it high} rate when one customer is active.
We find this surprising and look into it using dynamic programming.

\subsubsection{Dynamic programming}
\label{sec:DP}

Let $\left({I}(t), H(t)\right)$ be the Markov decision process with strategy $a$ on the state space $\mathcal{S}$, then $N-I(t)$ is the reward rate at time $t$. Define the maximal system efficiency
\begin{equation*}
	S^* \equiv  \limsup_{t\rightarrow \infty} \frac{\mathbb{E}_a \left[\int_{\tau = 0}^{t} (N-I(\tau)) \,d\tau \right]}{t}.
\end{equation*}

Our model is a controlled Markov jump process that evolves in continuous time, and in a discrete state space \citep[see][Chapter 5]{B11}. For $(i,h), \, (i',h') \in \mathcal{S}$, let $q_{(i,h),(i',h')}(a)$ be the transition rate from state $(i,h)$ to $(i',h')$ under strategy $a$. When $N = 3$, let $\Delta \geq \max_{(i,h)}\sum_{(i',h')} q_{(i,h),(i',h')}(a)$, for $(i,h)$, $i \in \{1,2,3\}, h \in \{0,1,\ldots, 3-i\}$. The optimality equation after the uniformization \citep[see][Propositions 5.3.1, 5.3.2]{B11} is
\begin{align} \label{eq:V1}
V(i,h) +  \frac{S^*}{\Delta} &  \, =	\, \frac{N-i}{\Delta}+ \frac{h \lambda_h}{\Delta} \, V(i+1,h-1)+ \frac{(N-i-h)\lambda_l}{\Delta} \, V(i+1,h)\\
	&+\max \left\{\underbrace{\frac{\mu_h}{\Delta} \, V(i-1,h+1) +\left(1-\frac{\mu_h+h\lambda_h+(N-i-h)\lambda_l}{\Delta}\right)V(i,h)}_{R_1(i,h)} \right.\,; \notag\\
	& \qquad \qquad \left.\underbrace{ \frac{\mu_l}{\Delta} \, V(i-1,h)  +\left(1-\frac{\mu_l+h\lambda_h+(N-i-h)\lambda_l}{\Delta}\right)V(i,h)}_{R_0(i,h)}\} \right\} \,,\notag
\end{align}
where $V: \mathcal{S} \rightarrow \mathbb{R}$ is the relative value function.
The values of $R_0(i,h)$ and $R_1(i,h)$ correspond to the future expected reward, where $R_0(i,h)$ relates to $a(i,h) = 0$ and $R_1(i,h)$ relates to $a(i,h) = 1$. 

States $(i,h)$, $i=0, h \in \{0,1,\ldots,N\}$ do not require any action and satisfy
\begin{small}$$V(0,h) + \frac{S^*}{\Delta} =
\frac{N}{\Delta}+\frac{h\lambda_h}{\Delta} V(1,h-1)+\frac{(N-h)\lambda_l}{\Delta} V(1,h) + \left(1-\frac{h\lambda_h + (N-h)\lambda_l}{\Delta}\right)V(0,h).
$$\end{small}
After excluding identical terms from $R_0(i,h)$ and $R_1(i,h)$, we have $R_0(i,h) > R_1(i,h)$ if and only if 
\begin{equation} \label{eq:inequality}
    \mu_l \, (V(i-1,h) - V(i,h)) > \mu_h \, (V(i-1,h+1) - V(i,h)) \,.
\end{equation}

From \eqref{eq:V1}, it can be seen that at each state $(i,h), \, i > 0$, a customer needs to consider three factors before making the decision: the value of next state $(i-1,h)$ or $(i-1,h+1)$, the probability to go to that state, and the probability to remain in the current state.

Now we select the parameter settings in the region in Figure \ref{fig:Optimalan3} where $00{*}|10|0$ is the optimal strategy, and set $\Delta = (3\lambda_h+3\lambda_l+\mu_h+\mu_l)$. Also, since $V(i,h)$ is a relative value, we can arbitrarily select a state, fix its relative value and calculate relative values for other states. Here we set $V(3,0) = 0$. Since $\left({I}(t), H(t)\right)$ is a Markov decision process, Equation \eqref{eq:V1} can be solved by methods such as value iteration or policy iteration. Here we solve it in a different way. Since we already know from the previous calculation that $000|10|0$ is the optimal strategy, we use it to easily compute the relative functions $V(i,h)$ given the strategy is $000|10|0$, by solving a set of 6 linear equations (with $V(3,0) = 0$), and then verify that $000|10|0$ is optimal.


\begin{table}[!ht]
\caption{Relative value functions when $\mu_h = 6.38, \lambda_h = 3.95, \mu_l = 1, \lambda_l = 0.5$ under the optimal strategy $000|{1}0|0$.} \label{tab:V}
\vspace*{10pt}
\centering
        {\def\arraystretch{1.5}  
{\footnotesize \begin{tabular*}{1\textwidth}{@{\extracolsep{\fill}}ccccccccccc}
\hline
\hline
$S^*$ & $V(0,0)$ & $V(0,1)$ & $V(0,2)$ & $V(0,3)$ &
$V(1,0)$ &
$V(1,1)$ &
$V(1,2)$ &
$V(2,0)$ &
$V(2,1)$ &
$V(3,0)$  \\ \hline
$0.5269$  & $1.2974$  &  $1.0002$  & $0.6287$  &  $0.1989$  & $0.9821$  & $0.5986$  & $0.1590$  & $0.5269$  & $0.0818$  & $0$\\
\hline
\hline
\end{tabular*}
}
}
\end{table}

\begin{figure}[h]
	\centering
	{\includegraphics[width=0.5\linewidth]{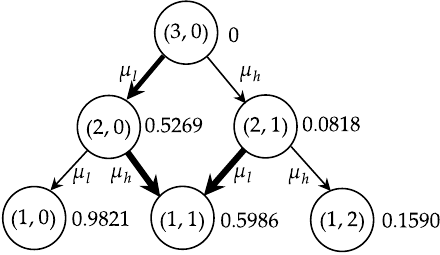}}
	\caption{The decision diagram with bold arrows denoting action for each state.}\label{fig:Decision}
\end{figure}
Table \ref{tab:V} represents the relative function values and the optimal system efficiency calculated for one numerical experiment. The decisions for states $(3,0)$, $(2,1)$ and $(2,0)$ and the states' relative values are depicted in Figure \ref{fig:Decision}. It can be seen that $V(1,0) = 0.9821 > 0.5986 = V(1,1)$.   
However, the probabilities to go to state $(1,0)$ and remain in the current state are $0.0482$ and $0.9518$, respectively if $a(2,0) = 0$. The probabilities to go to state $(1,1)$ and remain in the current state are $0.3078$ and $0.6922$, respectively, if $a(2,0) = 1$. Although $V(1,0) > V(1,1)$,
\[
V(1,0) - V(2,0) \quad < \quad 6.38 \, \, (V(1,1) - V(2,0)) \,.
\]
It can be checked that $R_0(i,h) > R_1(i,h)$ for the other states, so $000|10|0$ is optimal. Of course, $S^*$ equals what we obtained from the steady-state probabilities. 
As explained earlier, there may be more than one strategy that results in the same system efficiency. From our numerical experiments, $001|10|0$ does {\it not} satisfy Equation \eqref{eq:V1} because it prescribes a non-optimal choice for a transient state, but it still results in the same system efficiency as $000|10|0$ and is therefore also optimal.

\subsection{General $N$}

When $\mu_h/\lambda_h > \mu_l/\lambda_l$, it seems that it is better to always choose the high-rate service. However, the strategy also affects $\pi(i,h)$ and we need to consider the server idle probability. Thus, this is only true when the system is almost always busy, which is described in the following theorem. 
\begin{theorem} \label{con:1}
 When $N \rightarrow \infty$, the optimal strategy is to implement the more efficient service mode in any state.
\end{theorem}
\proof{Proof of Theorem \ref{con:1}}
    When $N \rightarrow \infty$, the server tends to always be busy, $\pi_{0,h}$ tends to zero $\forall h$, and
    \[
    \sum_{i = 1}^{N} \sum_{h=0}^{N-i} \, \pi^{(a)}_{i,h} \rightarrow 1  \,.
    \]
    It follows from Equation \eqref{eq:SW2} that to maximize $S(a)$, we need to maximize
    \[
    \sum_{i = 1}^{N} \sum_{h=0}^{N-i} \, \left(a(i,h) \frac{\mu_h}{\lambda_h} + (1-a(i,h)) \frac{\mu_l}{\lambda_l}\right) \, \pi^{(a)}_{i,h} \, \leq \,  \frac{\mu_r}{\lambda_r} \,,
    \]
    where $r = \arg\max_{q \in \{h,l\}} \, \cfrac{\mu_q}{\lambda_q}$.  
    Therefore we shall set $a(i,h) = \mathbbm{1}_{\{\frac{\mu_h}{\lambda_h} > \frac{\mu_l}{\lambda_l}\}}$ for any $i>0, h = 0,1,\ldots, N-i$.

We now introduce the definition of a threshold strategy. A threshold strategy with threshold $n \in \mathbb{N}$, has 
\begin{equation} \label{eq:threshold1}
	a(i,h) =
	\begin{cases}
		0 & N-i \geq n, \\
		1 & N-i < n.
	\end{cases} \,
\end{equation}
We will use $n$ to denote a threshold strategy, and $n^*$ to denote the threshold strategy that results in maximal (social) welfare.

In the $N=2$ case, for every fixed $\lambda_l, \mu_l$, and $\lambda_h$, when $\mu_h$ increases, the evolvement of the best strategy is
\[
0{*}|0  \rightarrow 00|1 \rightarrow {*}1|1 \,.
\]
We know that ${*}1|1$ and $0{*}|0$ result in the same system efficiency as $11|1$ and $00|0$, respectively. Strategies $00|0$, $00|1$, and $11|1$ are threshold strategies with threshold $n = 0,1,2$, respectively. This means if we only focus on threshold strategies, we are still be able to find the optimal social welfare for any parameter setting.
When $N = 3$, as described in Conjecture \ref{conjecture:1}, the evolution of the best strategy is either
\[
	0{**}|0{*}|0 \rightarrow 00{*}|00|1 \rightarrow 00{*}|10|1 \rightarrow 000|11|1 \rightarrow {**}1|{*}1|1
\]
or
\[
	0{**}|0{*}|0 \rightarrow 00{*}|10|0 \rightarrow 00{*}|10|1 \rightarrow 000|11|1  \rightarrow {**}1|{*}1|1 \,.
\]
Strategies $000|00|0$, $000|00|1$, $000|11|1$, and $111|11|1$ are threshold strategies with threshold $x = 0, 1,2,3$, respectively. However, strategies $00{*}|10|1$ and $00{*}|10|0$ are not of threshold type.

{\bf Remark}: When $N = 2,3$, if the slow and fast services are equally efficient, i.e. $\lambda_h/\mu_h = \lambda_l/\mu_l$, then the optimal strategy is of threshold type, and it is $n^* = N-1$.

\begin{table}[!ht]
\caption{For $\mu_l = 1, \, \lambda_l \in \{0.1,0.3,0.5,0.7,0.9\}, \, \mu_h \in [\mu_l+0.01, \, 4], \, \lambda_h \in [\lambda_l+0.01, \, 3]$, we calculate the relative maximal difference between the optimal system efficiency and the threshold strategy, and present the relative difference, which is $\displaystyle \frac{\max_{\lambda_h\in[\lambda_l+0.01,3], \, \mu_h\in[\mu_l+0.01,4]}\{\max_a S(a)-\max_{a \in \mathcal{T}} S(a)\}}{\max_{\lambda_h\in[\lambda_l+0.01,3], \, \mu_h\in[\mu_l+0.01,4]}\{\max_a S(a)\}}$.} \label{tab:VI}
\vspace*{10pt}
\centering
        {\def\arraystretch{1.5}  
\begin{tabular*}{0.8\textwidth}{@{\extracolsep{\fill}}cccccc}
\hline
\hline
		 \makecell[c]{$\lambda_l$ \\ $\mu_h$ \\ $\lambda_h$} & \makecell[c]{$0.1$ \\ $[1.01, 4]$ \\ $[0.11,3]$} & \makecell[c]{$0.3$ \\ $[1.01, 4]$ \\ $[0.31,3]$} & \makecell[c]{$0.5$ \\ $[1.01, 4]$ \\ $[0.51,3]$} & \makecell[c]{$0.7$ \\ $[1.01, 4]$ \\ $[0.71,3]$} & \makecell[c]{$0.9$ \\ $[1.01, 4]$ \\ $[0.91,3]$} \\
		 \hline
        $N = 3$ & $0.04\%$  & $ 0.21\%$  &  $0.27\%$  & $0.24\%$ & $0.14\%$  \\
		 \hline
	$N = 4$ & $0.09\%$  & $ 0.21\%$  &  $0.24\%$  & $0.22\%$ & $0.13\%$ \\
\hline
\hline
\end{tabular*}
\label{table:mechanical_property}
}
\end{table}

When there are $N$ customers in the system, there are totally $2^{N(N+1)/2}$ strategies, so the computation is very expensive if $N > 4$. For example, there are more than $3 \times 10^5$ strategies when $N =5$. If $\lambda_h \gg \lambda_l$, then the low rate is chosen in all states in the optimal strategy. If $\mu_h \gg \mu_l$, then the high rate is chosen in all states. Thus, we only look into the case when $\mu_h$ and $\lambda_h$ are not very much greater than $\mu_l$ and $\lambda_l$, respectively, for $N = 3,4$. Denote by $\mathcal{T}$ the set of threshold strategies. For $\mu_l < \mu_h \leq 4$, $\lambda_l < \lambda_h \leq 3$, $N = 3,4$, we compare the best strategy with the best threshold strategy, and present the relative maximal difference of the resulting system efficiency in Table \ref{tab:VI}. It can be seen that the relative difference
is always very small (at most $0.27\%$). In other words, {\it if we only focus on threshold strategies, the optimal system efficiency we can obtain will not differ much from that if we consider all the possible pure strategies.} Thus, when $N > 4$, we only consider a threshold strategy which can be denoted by an integer $n$, whose interpretation is that customers adopt the high rate when $i > n$, and take the low rate when $i \leq n$. There are totally $N+1$ threshold strategies to be considered. The calculation of the system efficiency under threshold $n$ is given in Appendix \ref{appendix:QS}.


\begin{figure}[ht]
    \centering
    \begin{minipage}{0.33\textwidth}
        \centering
        \includegraphics[width=\linewidth]{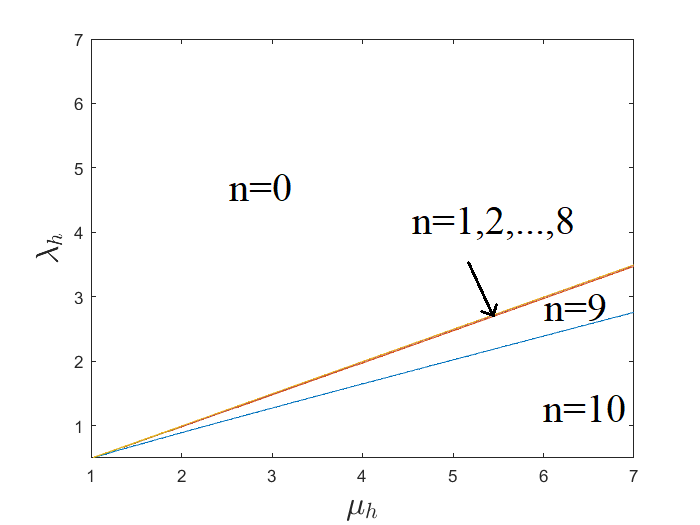}
        \textbf{(a)} 
        $\mu_l = 1, \, \lambda_l = 0.5$
    \end{minipage}%
    \begin{minipage}{0.33\textwidth}
        \centering
        \includegraphics[width=\linewidth]{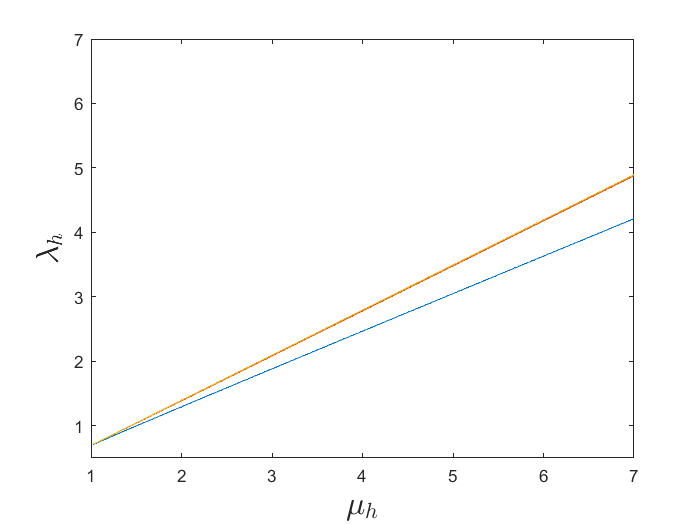}
        \textbf{(b)} 
        $\mu_l = 1, \, \lambda_l = 0.7$
    \end{minipage}%
    \begin{minipage}{0.33\textwidth}
        \centering
        \includegraphics[width=\linewidth]{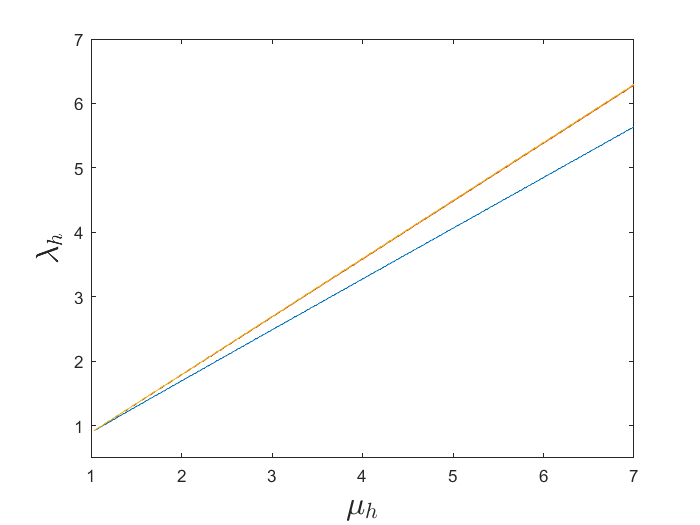}
        \textbf{(c)} 
        $\mu_l = 1, \, \lambda_l = 0.9$
    \end{minipage}
    \vspace*{10pt}
    \caption{The optimal strategy regions for different parameter settings when $N = 10$.}
    \label{fig:OptimalN10}
\end{figure}

An example of the optimal threhsold strategy for $N = 10$ is depicted in Figure \ref{fig:OptimalN10}. In this example, the region where the optimal strategy is $n = 1,2,3,\ldots,8$ is so small that it looks like a line (colored red), whose {\it slope} is, approximately, $\lambda_l/\mu_l$. This can be compared with Theorem \ref{con:1} that states that for large values of $N$ it is optimal to use the more efficient service type. 
\begin{table}[!ht]
\caption{The ratio of the change in efficiency when selecting the more efficient service type at every state to the optimal system efficiency, where the optimal system efficiency is over every pure strategy for $N = 2,3,4$ or over every threshold strategy for $N = 10,15$. The change is defined as the difference between the system efficiency if the more efficient service at every state is selected and the optimal system efficiency. The maximum overall searched values in the given intervals are presented.} \label{tab:VII}
\vspace*{10pt}
\centering
        {\def\arraystretch{1.5} 
\begin{tabular*}{0.8\textwidth}{@{\extracolsep{\fill}}c ccccc}
\hline
\hline
\multicolumn{2}{c}{$\mu_h$} & \multicolumn{4}{c}{$[3, 4]$} \\
\multicolumn{2}{c}{$\lambda_h$} & \multicolumn{4}{c}{$[\lambda_l+0.01, 1]$} \\
\Xhline{3\arrayrulewidth}
\multicolumn{2}{c}{$\lambda_l$} & 0.1 & 0.2 & 0.3 & 0.4 \\
\hline 
optimal  & $N = 2$ & $0.53\%$  & $1.50\%$  &  $2.21\%$  & $0$   \\
\multirow{2}{*}{strategy} & $N = 3$ & $1.12\%$ & $3.08\%$ & $4.10\%$  &  $0$   \\
& $N = 4$ & $1.76\%$ & $4.55\%$ & $5.05\%$  &  $0.55\%$  \\
\Xhline{3\arrayrulewidth}
\multicolumn{2}{c}{$\lambda_l$} & $0.05$ & $0.08$ & $0.1$ & $0.15$ \\
\hline
optimal & $N = 10$ & $1.50\%$ & $4.52\%$ & $5.55 \%$ & $4.18\%$ \\
{threshold strategy} & $N = 15$ & $3.68\%$ & $5.35\% $ & $2.93\%$ & $0.39\%$ \\
\hline
\hline
\end{tabular*}
\label{table:mechanical_property}
}
\end{table}

We compare the system efficiency when selecting the more efficient service type at every state, to the optimal (threshold) strategy, in Table \ref{tab:VII}.
When $N =2,3,4$, we use the optimal strategy among all pure strategies, when $N = 10, \, 15$, we use the optimal threshold strategy. We set $\mu_l = 1$ and choose $[3,4]$ and $[\lambda_l+0.01, 1]$ as the range for $\mu_h$ and $\lambda_h$, respectively, since usually the expected service time is much smaller than the active time. For example, an EV that is charged for 30 min can run for hours. We chose the values of  $\lambda_l$ that result in relatively larger ratios. 

The relative difference is not significant, but also not negligible. This is because there exists a parametric region where $n^* = N-1$. Take $N = 10$ for example, we have, from Figure \ref{fig:OptimalN10}, that $n^* = 9$ for some parameters. Strategy $n^* = N-1$ means choosing the slow service when there are no queueing customers (and one in service) is optimal although the fast service is more efficient. Based on the numerical examples we propose the following conjecture:
\begin{conjecture} \label{Conj:1}
When $N$ is large, if the slow service is more efficient then it is optimal to adopt it. If the fast service is more efficient, there are two cases. In the first case, if the efficiency difference is under some bound, then it is optimal to use the slow service when this is the only inactive customer, and use the fast service otherwise. In the second case, if the efficiency difference is greater than this bound, then it is always optimal to use the fast service. This also means that when $N$ is large, with the following strategies, we can approximate the optimal strategy well. 
    \begin{itemize}
        \item[$\circ$] always choose the slow service;
        \item[$\circ$] choose the slow service when there is only one inactive customer and the fast service otherwise;
        \item[$\circ$] always choose the fast service.
    \end{itemize}
\end{conjecture}
It can be observed from Figures \ref{fig:Optimalan2}, \ref{fig:Optimalan3}, and \ref{fig:OptimalN10} that the three threshold strategies almost occupy the whole parameter space.

\section{Nash equilibrium}  \label{sec:NE}

In this section, we are interested in the equilibrium strategy when customers maximize their own activity time. Unlike the case of overall optimization, we cannot exclude the possibility that the resulting strategies use randomization. We further assume that customers apply (mixed) threshold strategies, choosing long service when the queue length exceeds a threshold and perhaps randomizing at the threshold. This assumption simplifies customer strategies and can be justified by the {\it bounded rationality}, i.e., individuals cannot make complex calculations or follow complex strategies. For example, in our case, they cannot infer the service types of active customers from the information they possess.
Under full rationality, the equilibrium would differ, however, it is unrealistic to assume that individuals possess complete information and are capable of fully processing it.


In contrast to the threshold strategies used to optimize the system (see \eqref{eq:threshold1}), here a threshold strategy with threshold $x = n + p, \, n \in \mathbb{N},\, p \in [0,1)$, has
\begin{equation*} \label{eq:threshold2}
	a(i,h) =
	\begin{cases}
		1 & i \leq n, \\
		p & i = n+1, \\
		0 & i >  n+1.
	\end{cases} \,
\end{equation*}
In our numerical study, we always found at least one equilibrium of this type.

To work out the symmetric Nash equilibrium strategy, we first arbitrarily select a customer as our tagged customer, and assume that her strategy can be different from other customers' strategies.
Let $k = -1,0, \ldots, i$ represent the state of the tagged customer. When $k > 0$, $k$ is the position of the tagged customer in the queue; When $k = -1,0$, it means that the tagged customer is active at a low and high rate, respectively. Thus, $i$ does not include the tagged customer when $k = -1,0$, and includes her otherwise. The state of the system is represented by $(k,i,h)$, where $i$ and $h$ are the same as before; $i$ is the number of inactive customers, $h$ is the number of customers who are active at a high rate. We assume that the tagged customer uses a pure strategy $m \in \{0,1,\ldots, N\}$, and the other customers use a threshold strategy $x \in [0,N]$. The stationary distribution of $(k,i,h)$ is affected by $m$ and $x$, and is denoted by $\pi^{(m,x)}_{k,i,h}$.

With the stationary distribution, we can calculate the fraction of time that the tagged customer is active as
\[
U(m,x) =\sum_{k = 0,1} \sum_{i, h}\pi^{(m,x)}_{k,i,h} \,.
\]
The calculation of $\pi^{(m,x)}_{k,i,h}$ is given in Appendix \ref{appendix:QN}.

For any fixed value of $x$, we can calculate the best response of the tagged customer
\begin{equation*}
    b(x) = \arg\max_{m} U(m,x) \,.
\end{equation*}
Let $x_e$ be the Nash equilibrium threshold, then $x_e$ shall satisfy $x_e = b(x_e)$, or $x_e$ is between $b(x_e^-)$ and $b(x_e^+)$. In Figure \ref{fig:BR}, we plot the best responses for different $\lambda_h$ when $N = 5, \mu_l = 1, \mu_h = 3, \lambda_l = 0.9$, and line $b(x) = x$. The intersection points are $x_e$. It can be observed that the tagged customer's best-response to a threshold strategy adopted by others is a nondecreasing step function of the threshold. This behavior is called follow the crowd (FTC), and multiple equilibria are
possible when it is FTC. Figures \ref{fig:BR}(c) and 7(d) are examples of multiple equilibria.

\begin{figure}[h]
    \centering
    \begin{minipage}{0.32\textwidth}
        \centering
        \includegraphics[width=\linewidth]{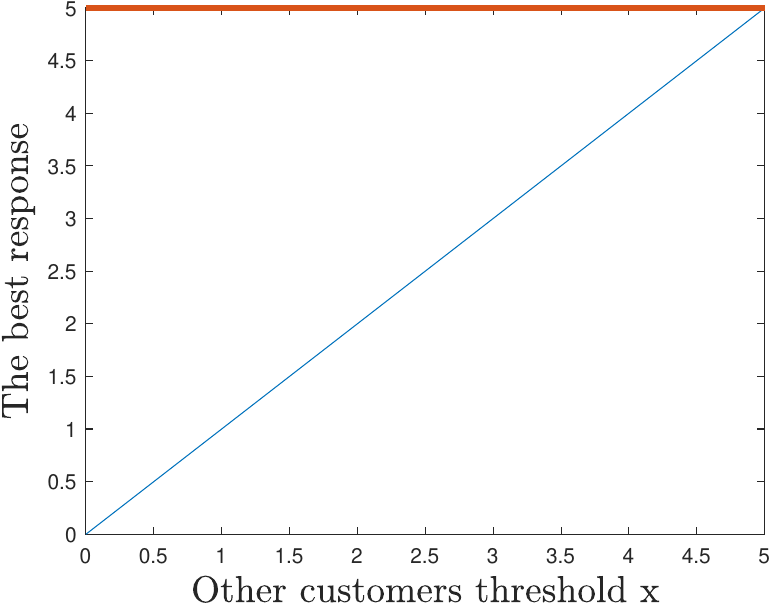}
        \textbf{(a)} \\
        $\lambda_h = 1.05, \, x_e = 5.$
    \end{minipage}%
    \begin{minipage}{0.32\textwidth}
        \centering
        \includegraphics[width=\linewidth]{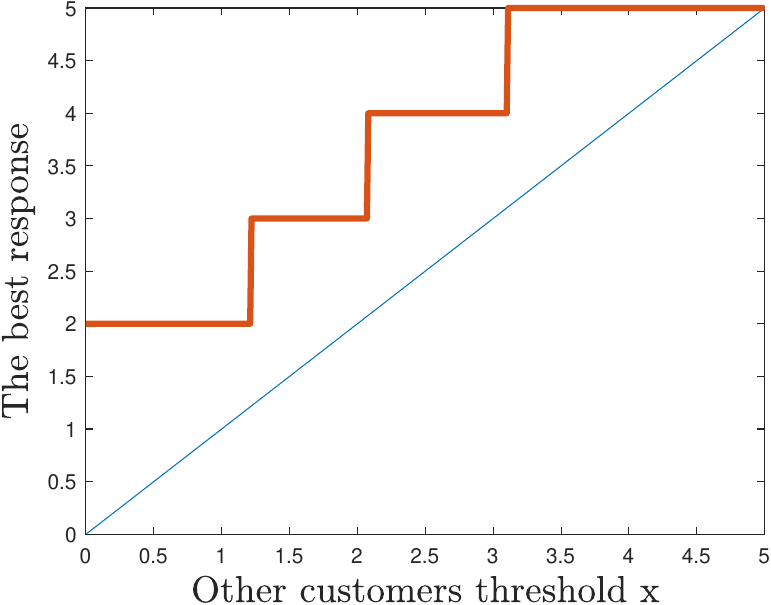}
        \textbf{(b)} \\
        $\lambda_h = 1.20, \, x_e = 5.$
    \end{minipage}%
    \begin{minipage}{0.32\textwidth}
        \centering
        \includegraphics[width=\linewidth]{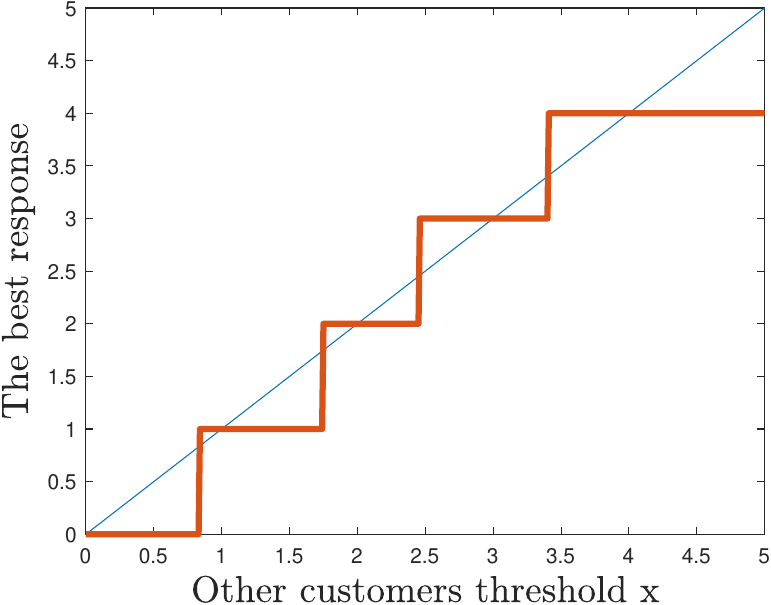}
        \textbf{(c)} \\
        $\lambda_h = 1.44, \, x_e = 0, 0.83, 1, 1.74, 2, 2.45, 3, 3.41, 4.$
    \end{minipage}

    \vspace{0.5em} 
    
    \begin{minipage}{0.32\textwidth}
        \centering
        \includegraphics[width=\linewidth]{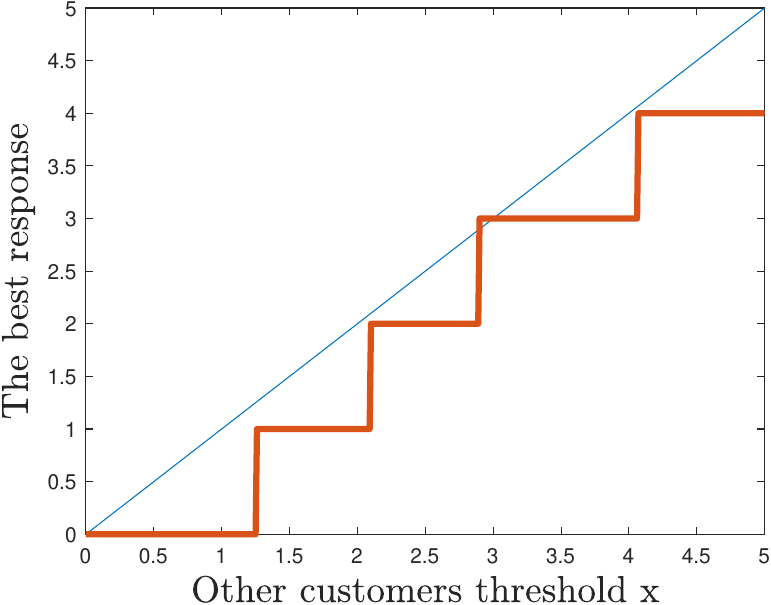}
        \textbf{(d)} \\
        $\lambda_h = 1.52, \, x_e = 0, 2.88, 3.$
    \end{minipage}%
    \begin{minipage}{0.32\textwidth}
        \centering
        \includegraphics[width=\linewidth]{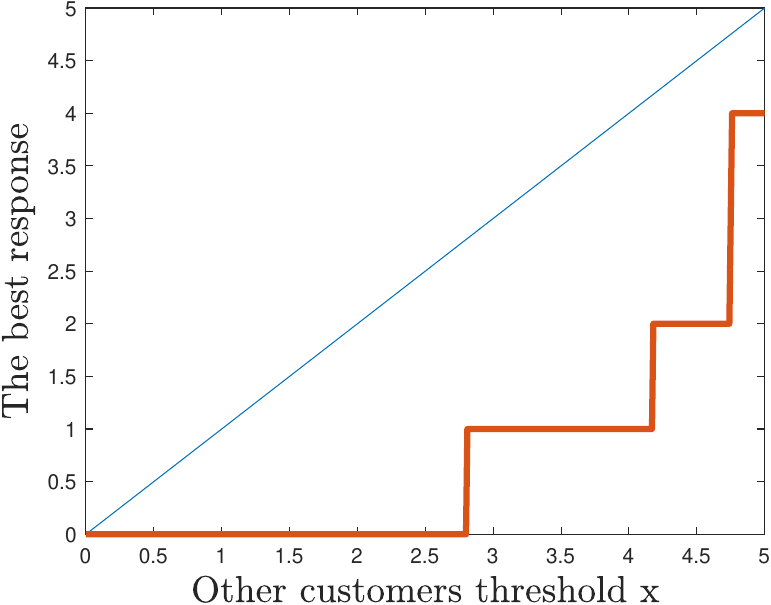}
        \textbf{(e)} \\
        $\lambda_h = 1.86, \, x_e = 0.$
    \end{minipage}%
    \begin{minipage}{0.32\textwidth}
        \centering
        \includegraphics[width=\linewidth]{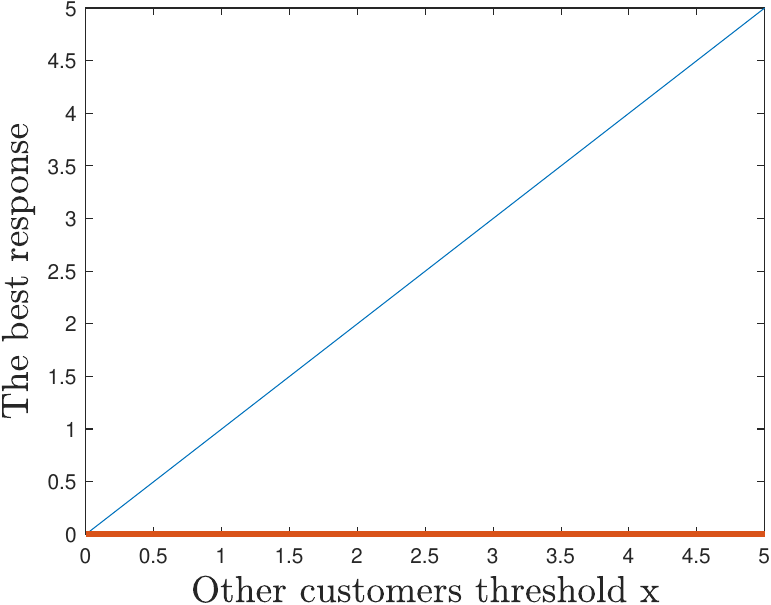}
        \textbf{(f)} \\
        $\lambda_h = 2.10, \, x_e = 0.$
    \end{minipage}
    
    \caption{The best response of the tagged customer for different values of other customers' threshold when $N = 5, \mu_l = 1, \mu_h = 3, \lambda_l = 0.9$.}
    \label{fig:BR}
\end{figure}

There is always a discrepancy between individual behavior and optimal strategy. We now check how inferior social welfare can be under the worst Nash equilibrium. We focus on threshold strategies depending on the number of inactive customers for both the Nash equilibria and the optimal strategy. 
To compare the efficiency of the system that resulted from the Nash equilibria with the optimal efficiency of the system $S^*$, we compare the efficiency of the system resulted from the Nash equilibria with the optimal efficiency of the system $S^*$. 
compute the price of anarchy
\begin{equation*}
    \text{PoA} \equiv \frac{\displaystyle\max_{a \in \mathcal{T}} S(a)}{\text{minimum of system efficiency under the Nash equilibria}} \,.
\end{equation*}
 
If both $n_e$ and $n_e'$ are equilibria, we also observe {numerically} that the efficiency of the system under $n_e$ is less if $n_e < n_e'$.
If the slow service is more efficient, that is, if $\mu_l/\lambda_l > \mu_h/\lambda_h$, then choosing the slow service is a dominant strategy. Therefore, if multiple Nash equilibria are observed, the fast service must be more efficient.

\begin{theorem} \label{thm:PoA}
  For fixed values of the parameters $\lambda_h,\lambda_l,\mu_h$ and $\mu_l$, when $N \rightarrow \infty$, the PoA may be arbitrarily large.
\end{theorem}

\proof{Proof of Theorem \ref{thm:PoA}.}
Suppose that $\mu_h = 1$, and $\mu_l = 1/K, \, K > 1$.  As $N \rightarrow \infty$, customers prefer the slow service because the expected waiting in the queue tends to $\infty$, and customers want to be active for a longer time. Thus, selecting the low rate is a dominant strategy.

The ratio of the efficiency of the high rate to the low rate is $K \, \lambda_l/\lambda_h$. If $K>\lambda_h/\lambda_l$, the high rate is more efficient. It follows from Theorem \ref{con:1} that the high-rate service is optimal. In this case, PoA$ = K \lambda_l/\lambda_h$. If we fix $\lambda_l/\lambda_h$, then PoA can be arbitrarily large by tuning the value of $K$. 

Examples of system efficiency in Nash equilibrium, under the optimal threshold strategy, and by only using the more efficient service, are presented in Table \ref{tab:poa}. It can be seen that when $N=30$, the optimal threshold policy is $n^* = N$ and the value of PoA is close to $K\lambda_h/\lambda_l$ as described in Theorem \ref{thm:PoA}. When $N = 10$, we picked parameters such that the optimal threshold policy is $n^* = N-1$, so PoA$\ne 1$ after regulation, but we can still decrease PoA by only using the more efficient service.
Note that if the slow service is more efficient, individuals’ optimal choice is to select the slow service. Furthermore, Conjecture~\ref{Conj:1} suggests that when \( N \) is large, it is optimal to adopt the slow service if it is more efficient. Numerical examples in Figure~\ref{fig:OptimalN10} indicate that this already holds when \( N = 10 \). Therefore, in Table~\ref{tab:poa}, the PoA is equal to 1 when slow service is more efficient.

\begin{table}[!ht]
\caption{The worst system efficiency under Nash equilibria, the optimal system efficiency under threshold strategies, and the value of PoA (original and after regulation) when $\mu_l = 1, \lambda_l = 0.2$.} \label{tab:poa}
\vspace*{10pt}
\centering
        {\def\arraystretch{1.5} 
{\footnotesize \begin{tabular*}{1\textwidth}{@{\extracolsep{\fill}}ccccccc}
\hline
\hline
 $N$ & $30$ & $30$ & $30$ & $30$ & $10$ & $10$ \\
    \hline
    $\mu_h$ & $2$ & $3$ & $4$ & $5$ & $2$ & $5$ \\
    \hline
    $\lambda_h$ & $0.21$ & $0.21$ & $0.21$ & $0.21$ & $0.35$ & $0.9$ \\
    \Xhline{3\arrayrulewidth}
    worst system efficiency under $n_e$ & $0.1667$ & $0.1667$ & $0.1667$ & $0.1667$ & $0.4908$ & $0.4908$\\
    \hline
    $\max_{a \in \mathcal{T}} S(a)$ & $0.3175$ & $ 0.4761$ & $0.6317$ & $0.7635$ & $0.5541$ & $0.5466$  \\
    \hline
    system efficiency by only using the more efficient service & $0.3175$  & $ 0.4761$ & $0.6317$  &  $0.7635$  &  $0.5515$  & $0.5385$ \\
    \hline
    PoA & $1.9046$ & $2.8560$ & $3.7894$ & $4.5801$ & $1.1290$ & $1.1096$ \\
    \hline
    PoA after regulation by only using the more efficient service & 1 & 1 & 1 & 1 & $1.0047$ & $1.0113$ \\
\hline
\hline
\end{tabular*}}
\label{table:mechanical_property}
}
\end{table}


\section{Concluding remarks}
\label{sec:FR}

This paper investigates efficiency optimization in a closed queueing network with two service types. 
The queueing mechanism adds an extra cost when the customer demands a new service, and the interactions among the customers emerge from the need to share a common queue.
We show that choosing the more efficient service is not always optimal, but the optimal strategy can be approximated well by selecting one of three threshold strategies which depend on the number
of inactive customers.
We also show that if customers make their choice in a selfish way, the system may be extremely inefficient.

A natural extension of this work would be to generalize the framework to accommodate multiple service types. Specifically, we could consider a system with three service rate categories: \textbf{fast}, \textbf{medium}, and \textbf{slow}, characterized by parameters $\lambda_l$, $\lambda_m$, $\lambda_h$, $\mu_l$, $\mu_m$, and $\mu_h$, where $\lambda_l < \lambda_m < \lambda_h$ and $\mu_l < \mu_m < \mu_h$.
Consider two individuals, each selecting one service type. The state of the system is characterized by the tuple $(i,j,k)$, where $i$ denotes the number of inactive customers, $j$ represents the number of active customers in fast service, $k$ indicates the number of active customers in medium service, and $N-i-j-k$ active customers in slow service.
The strategy space can be represented as:
\[
a(1,0,0)\,a(1,0,1)\,a(1,1,0) \mid a(2,0,0),
\]
where each $a(i,j,k) \in \{h,\,m,\,l \}$, represents a service rate choice, yielding $3^4 = 81$ possible strategies. 
Our numerical experiments with the parameters $\lambda_l = 0.9$, $\lambda_m = 1$, $\lambda_h = 2$, $\mu_l = 1$, $\mu_m = 5$, and $\mu_h$ increases from $8$ to $18$, show the evolution of the optimal strategy $mmm|h \rightarrow lll|h$. Across all tested parameter configurations, threshold-type strategies consistently emerge as optimal, with higher-rate services being adopted when more customers are inactive. This behavior is consistent with the two-type case. When $N > 2$, a similar analysis can be applied.

Regardless of the number of service types, when faster services are more efficient, choosing a slower service may still be individually optimal if others also select the slower service, as it allows one to avoid spending time waiting in a queue. Consequently, removing the slow service option can still be used to decrease the PoA.

Other extensions include heterogeneous customers, non-exponential service and activity duration, multiple service modes, or even a continuous interval of service types, and allowing customers to return to service before their active time expires. For example, an electric vehicle driver can decide to go to the charging station when there is 20\% battery and decide to leave before it is fully charged. 

We demonstrate how service efficiency affects periodic decisions when the cost in each period is endogenous and affected by the service type selection. Thus, we adopted only two types of services. The matrix analytic method (MAM) used in our work efficiently calculates the stationary distribution of two-dimensional Markov chains. It can also be used in three-dimensional Markov chains, as we did in Section 4. When there are more than two types of services, as long as the number is fixed, theoretically MAM can be adopted to solve it. For example, algorithms in \cite{DDKD20}, which deal with multidimensional level-dependent Markov chains, can be used to calculate the stationary distribution.

In this paper, we follow the common assumption of linear waiting costs.  Therefore, the time already spent waiting is considered a sunk cost. In real-life situations, the waiting cost function may be nonlinear, and after a long waiting period, the needs of the customers may become more urgent and affect the customer's choice. Adding this type of assumption will complicate the solution, which may also depend on the specific choice of the function. We leave this extension to future research.

\ACKNOWLEDGMENT{This research is supported by the Israel Science Foundation (Grants no. 1571/19 and 852/22).}


\bibliographystyle{pomsref}

 \let\oldbibliography\thebibliography
 \renewcommand{\thebibliography}[1]{%
    \oldbibliography{#1}%
    \baselineskip14pt 
    \setlength{\itemsep}{10pt}
 }
\ECSwitch 

\ECHead{APPENDICES}

\section{Stationary distributions when $N = 2$}
\label{appendix:1}

For strategy $0{*}|0$,
\begin{align*}
    & \pi_{0,0} = \frac{\mu _l^2}{2 \lambda _l \mu _l+2 \lambda _l^2+\mu _l^2}\,, \qquad  \pi_{0,1} = 0\,, \qquad \pi_{0,2} = 0\,, \\
    & \pi_{1,0} = \frac{2 \lambda _l \mu _l}{2 \lambda _l \mu _l+2 \lambda _l^2+\mu _l^2} \,, \qquad \pi_{1,1} = 0 \,, \qquad \pi_{2,0} = \frac{2 \lambda _l^2}{2 \lambda _l \mu _l+2 \lambda _l^2+\mu _l^2} \,. \notag
\end{align*}
For strategy ${*}1|1$,
\begin{align*}
    & \pi_{0,0} = 0 \,, \qquad \pi_{0,1} = 0 \,, \qquad \pi_{0,2} = \frac{\mu _h^2}{2 \lambda _h \mu _h+2 \lambda _h^2+\mu _h^2} \,, \\
    & \pi_{1,0} = 0 \,, \qquad \pi_{1,1} = \frac{2 \lambda _h \mu _h}{2 \lambda _h \mu _h+2 \lambda _h^2+\mu _h^2} \,, \qquad \pi_{2,0} = \frac{2 \lambda _h^2}{2 \lambda _h \mu _h+2 \lambda _h^2+\mu _h^2} \,. \notag
\end{align*}
For strategy $00|1$,
\begin{align*}
    & \pi_{0,0} = \frac{\lambda _h \mu _h \mu _l^2}{2 \lambda _l^2 \left(\lambda _h+\mu _h\right) \left(\lambda _h+\lambda _l\right)+2 \lambda _l \mu _l \left(\mu _h \lambda _l+\lambda _h \left(\mu _h+\lambda _l\right)\right)+\lambda _h \mu _h \mu _l^2} \,, \\
    & \pi_{0,1} = \frac{2 \mu _h \lambda _l^2 \mu _l}{2 \lambda _l^2 \left(\lambda _h+\mu _h\right) \left(\lambda _h+\lambda _l\right)+2 \lambda _l \mu _l \left(\mu _h \lambda _l+\lambda _h \left(\mu _h+\lambda _l\right)\right)+\lambda _h \mu _h \mu _l^2} \,, \notag \\
    & \pi_{0,2} = 0 \,, \notag\\
    & \pi_{1,0} = \frac{2 \lambda _h \mu _h \lambda _l \mu _l}{2 \lambda _l^2 \left(\lambda _h+\mu _h\right) \left(\lambda _h+\lambda _l\right)+2 \lambda _l \mu _l \left(\mu _h \lambda _l+\lambda _h \left(\mu _h+\lambda _l\right)\right)+\lambda _h \mu _h \mu _l^2} \,, \notag \\
    & \pi_{1,1} = \frac{2 \mu _h \lambda _l^2 \left(\lambda _h+\lambda _l\right)}{2 \lambda _l^2 \left(\lambda _h+\mu _h\right) \left(\lambda _h+\lambda _l\right)+2 \lambda _l \mu _l \left(\mu _h \lambda _l+\lambda _h \left(\mu _h+\lambda _l\right)\right)+\lambda _h \mu _h \mu _l^2} \,, \notag \\
    & \pi_{2,0} = \frac{2 \lambda _h \lambda _l^2 \left(\lambda _h+\lambda _l+\mu _l\right)}{2 \lambda _l^2 \left(\lambda _h+\mu _h\right) \left(\lambda _h+\lambda _l\right)+2 \lambda _l \mu _l \left(\mu _h \lambda _l+\lambda _h \left(\mu _h+\lambda _l\right)\right)+\lambda _h \mu _h \mu _l^2} \,. \notag
\end{align*}
For strategy $11|0$,
\begin{align*}
    & \pi_{0,0} = 0\,, \\
    & \pi_{0,1} = \frac{2 \lambda _h^2 \mu _h \mu _l}{2 \lambda _h^3 \left(\lambda _l+\mu _l\right)+2 \lambda _h^2 \left(\lambda _l+\mu _l\right) \left(\mu _h+\lambda _l\right)+2 \lambda _h \mu _h \lambda _l \mu _l+\mu _h^2 \lambda _l \mu _l} \,, \notag \\
    & \pi_{0,2} = \frac{\mu _h^2 \lambda _l \mu _l}{2 \lambda _h^3 \left(\lambda _l+\mu _l\right)+2 \lambda _h^2 \left(\lambda _l+\mu _l\right) \left(\mu _h+\lambda _l\right)+2 \lambda _h \mu _h \lambda _l \mu _l+\mu _h^2 \lambda _l \mu _l} \,, \notag\\
    & \pi_{1,0} = \frac{2 \lambda _h^2 \mu _l \left(\lambda _h+\lambda _l\right)}{2 \lambda _h^3 \left(\lambda _l+\mu _l\right)+2 \lambda _h^2 \left(\lambda _l+\mu _l\right) \left(\mu _h+\lambda _l\right)+2 \lambda _h \mu _h \lambda _l \mu _l+\mu _h^2 \lambda _l \mu _l} \,, \notag \\
    & \pi_{1,1} = \frac{2 \lambda _h \mu _h \lambda _l \mu _l}{2 \lambda _h^3 \left(\lambda _l+\mu _l\right)+2 \lambda _h^2 \left(\lambda _l+\mu _l\right) \left(\mu _h+\lambda _l\right)+2 \lambda _h \mu _h \lambda _l \mu _l+\mu _h^2 \lambda _l \mu _l} \,, \notag \\
    & \pi_{2,0} = \frac{2 \lambda _h^2 \lambda _l \left(\lambda _h+\mu _h+\lambda _l\right)}{2 \lambda _h^3 \left(\lambda _l+\mu _l\right)+2 \lambda _h^2 \left(\lambda _l+\mu _l\right) \left(\mu _h+\lambda _l\right)+2 \lambda _h \mu _h \lambda _l \mu _l+\mu _h^2 \lambda _l \mu _l} \,. \notag
\end{align*}
For strategy $10|0$,
\begin{align*}
    & \pi_{0,0} = 0 \,, \qquad \pi_{0,1} = \frac{\mu _h \mu _l \left(\lambda _h+\mu _l\right)}{\left(\lambda _l+\mu _l\right) \left(\lambda _h \left(\lambda _h+\mu _h+\lambda _l\right)+\mu _l \left(\lambda _h+\mu _h\right)\right)} \,,\\
    & \pi_{0,2} = 0 \,, \qquad \pi_{1,0} = \frac{\lambda _h \mu _l \left(\lambda _h+\lambda _l+\mu _l\right)}{\left(\lambda _l+\mu _l\right) \left(\lambda _h \left(\lambda _h+\mu _h+\lambda _l\right)+\mu _l \left(\lambda _h+\mu _h\right)\right)} \,, \notag\\
    & \pi_{1,1} = \frac{\mu _h \lambda _l \mu _l}{\left(\lambda _l+\mu _l\right) \left(\lambda _h \left(\lambda _h+\mu _h+\lambda _l\right)+\mu _l \left(\lambda _h+\mu _h\right)\right)} \,, \notag\\
    & \pi_{2,0} = \frac{\lambda _h \lambda _l \left(\lambda _h+\mu _h+\lambda _l+\mu _l\right)}{\left(\lambda _l+\mu _l\right) \left(\lambda _h \left(\lambda _h+\mu _h+\lambda _l\right)+\mu _l \left(\lambda _h+\mu _h\right)\right)} \,. \notag
\end{align*}
For strategy $10|1$,
\begin{align*}
    & \pi_{0,0} = 0 \,, \qquad \pi_{0,1} = \frac{\mu _h \mu _l \left(\mu _h+\lambda _l\right)}{\left(\lambda _h+\mu _h\right) \left(\left(\lambda _l+\mu _l\right) \left(\mu _h+\lambda _l\right)+\lambda _h \lambda _l\right)} \,,\\
    & \pi_{0,2} = 0 \,, \qquad \pi_{1,0} = \frac{\lambda _h \mu _h \mu _l}{\left(\lambda _h+\mu _h\right) \left(\left(\lambda _l+\mu _l\right) \left(\mu _h+\lambda _l\right)+\lambda _h \lambda _l\right)} \,, \notag\\
    & \pi_{1,1} = \frac{\mu _h \lambda _l \left(\lambda _h+\mu _h+\lambda _l\right)}{\left(\lambda _h+\mu _h\right) \left(\left(\lambda _l+\mu _l\right) \left(\mu _h+\lambda _l\right)+\lambda _h \lambda _l\right)} \,, \notag\\
    & \pi_{2,0} = \frac{\lambda _h \lambda _l \left(\lambda _h+\mu _h+\lambda _l+\mu _l\right)}{\left(\lambda _h+\mu _h\right) \left(\left(\lambda _l+\mu _l\right) \left(\mu _h+\lambda _l\right)+\lambda _h \lambda _l\right)} \,. \notag
\end{align*}

\section{Proof of {strategy $10|1$  is never optimal}} \label{appendix:101}

    The proofs for $11|0, 10|0$ and $10|1$ are similar. \underline{To prove that $10|1$ is never optimal}, we calculate
{\small \begin{align*} 
	&S(10|1)-S(00|1) = \\
	& \frac{\mu _h \lambda _h \mu _l \, \delta_1}{\left(\lambda _h+\mu _h\right) \left(2 \lambda _l^2 \left(\lambda _h+\mu _h\right) \left(\lambda _h+\lambda _l\right)+2 \lambda _l \mu _l \left(\mu _h \lambda _l+\lambda _h \left(\mu _h+\lambda _l\right)\right)+\lambda _h \mu _h \mu _l^2\right) \left(\left(\lambda _l+\mu _l\right) \left(\mu _h+\lambda _l\right)+\lambda _h \lambda _l\right)} \,, \notag
\end{align*}}
where
{\small \begin{align*}
	\delta_1 = -2 \lambda _h^2 \lambda _l \mu _l +\lambda _h \left(2 \mu _h \lambda _l^2-3 \mu _h \lambda _l \mu _l-\mu _l^2 \left(\mu _h+2 \lambda _l\right)-2 \lambda _l^2 \mu _l\right) + 2 \mu _h \lambda _l^3+2 \mu _h^2 \lambda _l^2+3 \mu _h \lambda _l^2 \mu _l+\mu _h^2 \lambda _l \mu _l\,,
\end{align*}
}and
\begin{align*} 
	S(10|1)-S({*}1|1) =\frac{\mu _h \, \delta_2}{\left(\lambda _h+\mu _h\right) \left(2 \lambda _h \mu _h+2 \lambda _h^2+\mu _h^2\right) \left(\left(\lambda _l+\mu _l\right) \left(\mu _h+\lambda _l\right)+\lambda _h \lambda _l\right)}
\end{align*}
where
\begin{equation*}
	\delta_2 = 2 \lambda _h^3 \mu _l+\lambda _h^2 \left(4 \mu _h \mu _l -2 \mu _h \lambda _l +2 \lambda _l \mu _l\right)-\lambda _h \left(2 \mu _h \lambda _l^2+3 \mu _h^2 \lambda _l-\mu _h^2 \mu _l\right) - \mu _h^3 \lambda _l-\mu _h^2 \lambda _l^2\,.
\end{equation*}
The signs of $S(10|1)-S(00|1)$ and $S(10|1)-S({*}1|1)$ are decided by $\delta_1$ and $\delta_2$, respectively. We treat $\delta_1$ and $\delta_2$ as the polynomial of $\lambda_h$ and arrange them in the descending order. It can be seen that for any values of $\lambda_l,\mu_h,\mu_l$, $ 2 \mu _h \lambda _l^3+2 \mu _h^2 \lambda _l^2+3 \mu _h \lambda _l^2 \mu _l+\mu _h^2 \lambda _l \mu _l > 0$ and $S(10|1)-S(00|1) =0$ when $\lambda_h = 0$. The value of $\delta_1$ is  increasing in $\lambda_h$ first and then decreasing in $\lambda_h$ to $-\infty$. Thus, except for $\lambda_h = 0$, there exists a unique $\tilde{\lambda}_h  > 0$ such that $S(10|1)-S(00|1) > 0$ when $\lambda_h < \tilde{\lambda}_h$, and $S(10|1)-S(00|1) < 0$ when $\lambda_h > \tilde{\lambda}_h$. So there is a unique value $\tilde{\lambda}_h$ that makes $S(10|1) = S(00|1)$. Similarly, $S(10|1) < S(*1|1)$ when $\lambda_h = 0$. As $\lambda_h$ increases, $\delta_2$ is either increasing from a negative value to $\infty$, or decreasing first, from a negative value, and then increasing to $\infty$. In either case, there exists a unique value of $\lambda_h \geq 0$ such that $S(10|1)=S({*}1|1)$.


\begin{figure}[h]
\centering
	\includegraphics[width=0.6\linewidth]{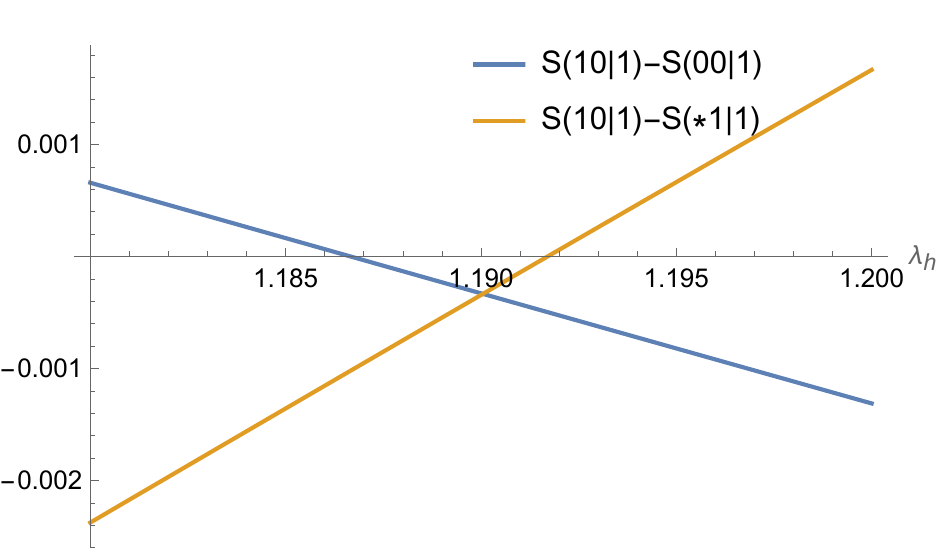}
	\caption{$\mu _h = 1.8,\,\lambda_l= 0.7$.}	\label{fig:delta1}
\end{figure}
The value of $\lambda_h$ that makes $S(10|1) = S(00|1)$ is
{\scriptsize \begin{align*}
	& \tilde{\lambda}_h = \\
	& \frac{2 \lambda _l^2 \left(\mu _h-\mu _l\right)-\lambda _l \mu _l \left(3 \mu _h+2 \mu _l\right)-\mu _h \mu _l^2+\sqrt{8 \mu _h \lambda _l^2 \mu _l \left(2 \lambda _l \left(\mu _h+\lambda _l\right)+\mu _l \left(\mu _h+3 \lambda _l\right)\right)+\left(2 \lambda _l^2 \left(\mu _l-\mu _h\right)+\lambda _l \mu _l \left(3 \mu _h+2 \mu _l\right)+\mu _h \mu _l^2\right){}^2}}{4 \lambda _l \mu _l} \notag \,.
\end{align*}}Next we compare $S(*1|1)$ and $S(00|1)$ when $\lambda_h = \tilde{\lambda}_h$. We have
{\footnotesize
\begin{align*}
	S({*}1|1)-S(00|1) = \frac{2 \mu _h \, \delta_3}{\left(2 \lambda _h \mu _h+2 \lambda _h^2+\mu _h^2\right) \left(2 \lambda _l^2 \left(\lambda _h+\mu _h\right) \left(\lambda _h+\lambda _l\right)+2 \lambda _l \mu _l \left(\mu _h \lambda _l+\lambda _h \left(\mu _h+\lambda _l\right)\right)+\lambda _h \mu _h \mu _l^2\right)} \,,
\end{align*}
}where
\begin{equation*}
	\delta_3 =  \mu _h \lambda _l^2 \left(2 \lambda _h+\mu _h\right) \left(\lambda _h+\lambda _l\right)+\lambda _h \lambda _l \mu _l \left(\mu _h^2-2 \lambda _h \left(\lambda _h+\lambda _l\right)\right)-\lambda _h^2 \mu _l^2 \left(2 \lambda _h+\mu _h\right) \,.
\end{equation*}
The sign of $S({*}1|1)-S(00|1)$ is decided by $\delta_3$. If we substitute $\tilde{\lambda}_h$ into $\delta_3$, we have
\begin{align*}
	\delta_3 \mid_{\lambda_h = \tilde{\lambda}_h}= \frac{1}{8 \lambda _l^3} \, \left(\mathcal{E}_1-\mathcal{E}_2\right)  \label{eq:del1}
\end{align*}
where
\begin{align}
	\mathcal{E}_1 =
	& \mu _l^5 \left(12 \mu _h \lambda _l^2+6 \mu _h^2 \lambda _l+\mu _h^3+8 \lambda _l^3\right) + \mu _l^4 \left(68 \mu _h \lambda _l^3+44 \mu _h^2 \lambda _l^2+9 \mu _h^3 \lambda _l+32 \lambda _l^4\right)  \notag \\
	& +\mu _l^3 \left(140 \mu _h \lambda _l^4+120 \mu _h^2 \lambda _l^3+30 \mu _h^3 \lambda _l^2+40 \lambda _l^5\right) + \mu _l^2 \left(116 \mu _h \lambda _l^5+132 \mu _h^2 \lambda _l^4+39 \mu _h^3 \lambda _l^3+16 \lambda _l^6\right) \notag \\
	& +\mu _l \left(24 \mu _h \lambda _l^6+38 \mu _h^2 \lambda _l^5+11 \mu _h^3 \lambda _l^4\right) +  8 \mu _h^2 \lambda _l^6+6 \mu _h^3 \lambda _l^5 \,, \notag \\
	\mathcal{E}_2 = & \mu _h \lambda _l^3 \left(\mu _h+4 \lambda _l\right)+2 \lambda _l^2 \mu _l \left(11 \mu _h \lambda _l+5 \mu _h^2+4 \lambda _l^2\right)+6 \lambda _l \mu _l^2 \left(\mu _h+\lambda _l\right) \left(\mu _h+2 \lambda _l\right)+\mu _l^3 \left(\mu _h+2 \lambda _l\right){}^2 \notag \\
	&\sqrt{8 \mu _h \lambda _l^2 \mu _l \left(2 \lambda _l \left(\mu _h+\lambda _l\right)+\mu _l \left(\mu _h+3 \lambda _l\right)\right)+\left(2 \lambda _l^2 \left(\mu _l-\mu _h\right)+\lambda _l \mu _l \left(3 \mu _h+2 \mu _l\right)+\mu _h \mu _l^2\right){}^2} \notag\,.
\end{align}
We can show that
\begin{align}
	\left(\mathcal{E}_1\right)^2 - \left(\mathcal{E}_2\right)^2= 16 \mu _h^2 \lambda _l^7 \left(\mu _h-\mu _l\right){}^2 \left(2 \lambda _l+\mu _l\right) \left(\mu _h+2 \lambda _l\right) \left(\lambda _l^2 \left(\mu _h+2 \mu _l\right)+\lambda _l \mu _l \left(\mu _h+4 \mu _l\right)+\mu _h \mu _l^2\right) > 0 \notag \,.
\end{align}
Since $\mathcal{E}_1, \mathcal{E}_2 > 0$, this means $\mathcal{E}_1 > \mathcal{E}_2$. So $S(*1|1) > S(00|1)$ when $\lambda_h = \tilde{\lambda}_h$. That is, when $S(10|1) = S(00|1)$, $S({*}1|1) > S(00|1)$. Combining the properties of $S(10|1) - S(00|1)$ and $S(10|1) - S({*}1|1)$ we discussed earlier, for any fixed $\lambda_l,\mu_h,\mu_l$, as $\lambda_h$ increases, $S(10|1) - S(00|1)$ and $S(10|1) - S({*}1|1)$ intersect when are both negative. This is depicted in Figure \ref{fig:delta1}. Thus, for any fixed $\lambda_l,\mu_h,\mu_l$, if we alter the value of $\lambda_h$,
\begin{itemize}
	\item when $S(10|1) > S(00|1)$, $S({*}1|1) > S(10|1)$,
	\item when $S(10|1) > S({*}1|1)$, $S(00|1) > S(10|1)$. 
\end{itemize}

\section{Proof of {strategy 100 is never optimal}.}\label{appendix:100}
    \underline{To prove that $10|0$ is never optimal}, we calculate
	\begin{equation*} \label{eq:3}
		S(10|0)-S(0{*}|0) =\frac{\mu _l \, \delta_1}{\left(\lambda _l+\mu _l\right) \left(2 \lambda _l \mu _l+2 \lambda _l^2+\mu _l^2\right) \left(\lambda _h \left(\lambda _h+\mu _h+\lambda _l\right)+\mu _l \left(\lambda _h+\mu _h\right)\right)} \,,
	\end{equation*}
where
\begin{equation*}
	\delta_1 = -\lambda _h^2 \mu _l \left(2 \lambda _l+\mu _l\right) + \lambda _h \left(2 \lambda _l^2 \left(\mu _h-\mu _l\right)-3 \lambda _l \mu _l^2-\mu _l^3\right)+\mu _h \lambda _l \left(4 \lambda _l \mu _l+2 \lambda _l^2+\mu _l^2\right) \,,
\end{equation*}
and
\begin{equation*} \label{eq:4}
	S(10|0)-S(10|1) = \frac{\lambda _h \lambda_l \, \delta_2}{\left(\lambda _h+\mu _h\right) \left(\lambda _l+\mu _l\right) \left(\lambda _h \left(\lambda _h+\mu _h+\lambda _l\right)+\mu _l \left(\lambda _h+\mu _h\right)\right) \left(\left(\lambda _l+\mu _l\right) \left(\mu _h+\lambda _l\right)+\lambda _h \lambda _l\right)} \,,
\end{equation*}
where
\begin{align*}
    \delta_2 = & \, \lambda _h^3 \mu _l +\lambda _h^2 \left(\mu _l \left(\mu _h+\lambda _l\right)-\mu _h \lambda _l+\mu _h \mu _l+\lambda _l \mu _l+2 \mu _l^2\right) \notag\\
    & + \lambda _h \left(\lambda _l \mu _l \left(\mu _h+\lambda _l\right)+\mu _h \mu _l \left(\mu _h+\lambda _l\right)-\mu _h \mu _l \left(\mu _h+2 \lambda _l\right)-2 \mu _h^2 \lambda _l-2 \mu _h \lambda _l^2+3 \mu _h \mu _l^2+2 \lambda _l \mu _l^2+\mu _l^3\right) \notag \\
    &+ -\mu _h^3 \lambda _l-2 \mu _h^2 \lambda _l^2-\mu _h \lambda _l^3-\mu _h^2 \mu _l \left(\mu _h+2 \lambda _l\right)-\mu _h \lambda _l \mu _l \left(\mu _h+2 \lambda _l\right)+\mu _h \mu _l^3 \,.
\end{align*}

\begin{figure}[h]
\centering
	\includegraphics[width=0.6\linewidth]{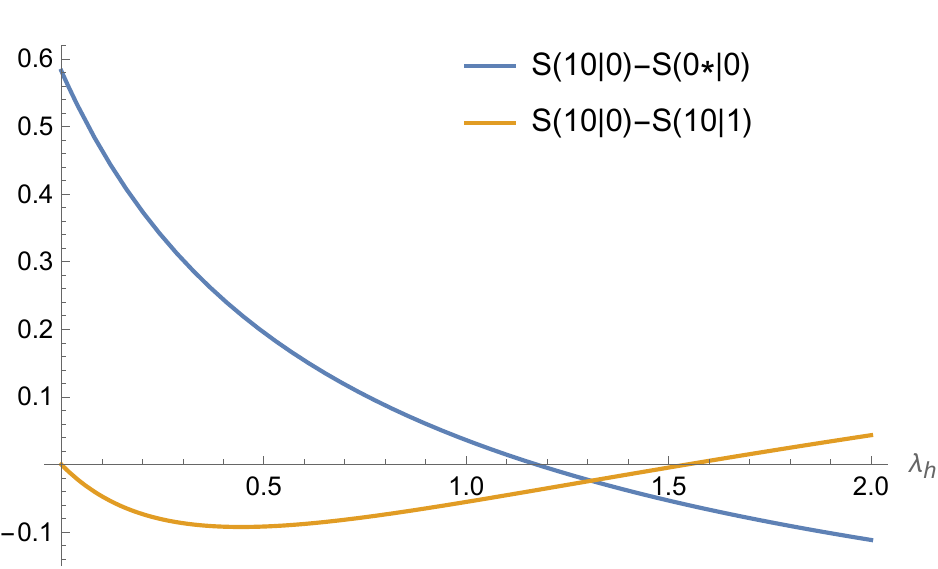}
	\caption{$\mu _h = 1.8,\,\lambda_l= 0.7$.}	\label{fig:delta2}
\end{figure}
Similarly, we treat $\delta_1$ and $\delta_2$ as the polynomial of $\lambda_h$ and arrange them in the descending order. It can be seen that for any values of $\lambda_l,\mu_h,\mu_l$, $\delta_1 = \mu _h \lambda _l \left(4 \lambda _l \mu _l+2 \lambda _l^2+\mu _l^2\right)$ when $\lambda_h = 0$. The value of $\delta_1$ either decreases in $\lambda_h$ to $-\infty$, or increases in $\lambda_h$ first and then decreases in $\lambda_h$ to $-\infty$. Thus, there exists a unique $\tilde{\lambda}_h$ such that $S(10|0)-S(0{*}|0) > 0$ when $\lambda_h < \tilde{\lambda}_h$, and $S(10|0)-S(0{*}|0) < 0$ when $\lambda_h > \tilde{\lambda}_h$. Similarly, there exists a unique $\tilde{\lambda}_h > 0$ such that $S(10|0)=S(10|1)$ when $\lambda_h = \tilde{\lambda}_h$.

Moreover, for any fixed $\lambda_l,\mu_h,\mu_l$, when $\lambda_h = \lambda_l\,\mu_h/\mu_l$,
\begin{equation*}
	\delta_1 = -(\mu_h-\mu_l)\, \mu _h \lambda _l^2 < 0\,, \qquad \delta_2 = -\left(\mu _h^2-\mu_l^2\right) \left(\lambda _l+\mu _l\right) \mu _h  < 0 \,.
\end{equation*}
This is the case only if $S(10|0)-S(0{*}|0)$ and $S(10|0)-S(10|1)$ intersect when they are negative, as depicted in Figure \ref{fig:delta2}. Thus, if we alter the value of $\lambda_h$, for any fixed $\lambda_l,\mu_h,\mu_l$,
\begin{itemize}
	\item when $S(10|0) > S(10|1)$, $S(0{*}|0) > S(10|0)$,
	\item when $S(10|0) > S(0{*}|0)$, $S({*}1|1) > S(10|0)$. 
\end{itemize}

\section{Proof of Lemma \ref{lemma:110}} \label{appendix:110}
    \underline{To prove that $11|0$ is never optimal}, we calculate
	{\footnotesize
	\begin{align*} \label{eq:5}
		&S(11|0)-S(10|0) =  \\
		& \frac{\mu _l \mu _h \lambda _l \, \delta_1}{\left(\lambda _l+\mu _l\right) \left(\lambda _h \left(\lambda _h+\mu _h+\lambda _l\right)+\mu _l \left(\lambda _h+\mu _h\right)\right) \left(2 \lambda _h^3 \left(\lambda _l+\mu _l\right)+2 \lambda _h^2 \left(\lambda _l+\mu _l\right) \left(\mu _h+\lambda _l\right)+2 \lambda _h \mu _h \lambda _l \mu _l+\mu _h^2 \lambda _l \mu _l\right)} \,, \notag
	\end{align*}
	}where
	{\small \begin{equation*}
		\delta_1 = -2 \lambda _h^3 \mu _l+\lambda _h^2 \left(2 \mu _h \lambda _l-3 \mu _h \mu _l-2 \lambda _l \mu _l-2 \mu _l^2\right)+\lambda _h \left(2 \mu _h \lambda _l^2+2 \mu _h^2 \lambda _l+3 \mu _h \lambda _l \mu _l-\mu _h \mu _l^2\right)+\mu _h^2 \lambda _l \mu _l \,.
	\end{equation*}
	}It can be seen that $\delta_1$ either increases from $\mu _h^2 \lambda _l \mu _l$ and decreases to $-\infty$, or decreases from $\mu _h^2 \lambda _l \mu _l$ to $-\infty$. In any case, there is a unique value of ${\lambda}_h > 0$ that makes $S(11|0)= S(10|0)$.
	
	Then we have \begin{small}
	\begin{equation*} \label{eq:6}
		S(11|0)-S({*}1|1) = \frac{2 \lambda _h^2 \, \delta_2}{\left(2 \lambda _h \mu _h+2 \lambda _h^2+\mu _h^2\right) \left(2 \lambda _h^3 \left(\lambda _l+\mu _l\right)+2 \lambda _h^2 \left(\lambda _l+\mu _l\right) \left(\mu _h+\lambda _l\right)+2 \lambda _h \mu _h \lambda _l \mu _l+\mu _h^2 \lambda _l \mu _l\right)} \,,
	\end{equation*}
	\end{small}
	where
	\begin{small}
	\begin{equation*}
		\delta_2 = 2 \lambda _h^3 \mu _l+\lambda _h^2 \left(-2 \mu _h \lambda _l+4 \mu _h \mu _l+2 \lambda _l \mu _l\right)+\lambda _h \left(-2 \mu _h \lambda _l^2-4 \mu _h^2 \lambda _l+2 \mu _h \lambda _l \mu _l+\mu _h^2 \mu _l\right) -2 \mu _h^3 \lambda _l-2 \mu _h^2 \lambda _l^2+\mu _h^2 \lambda _l \mu _l\,.
	\end{equation*}
	\end{small}
	\begin{figure}[h]
    \centering
	\includegraphics[width=0.6\linewidth]{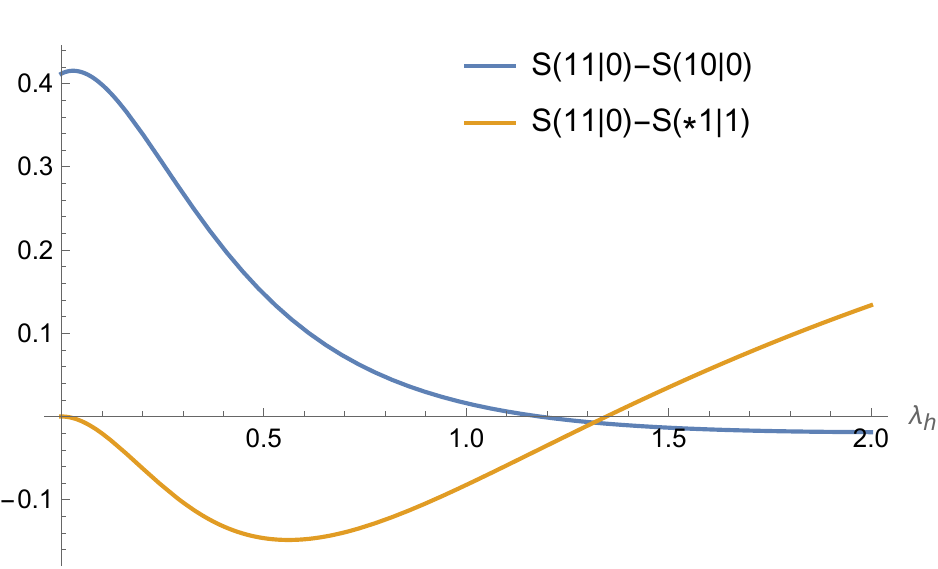}
	\caption{$\mu _h = 1.8,\,\lambda_l= 0.7$.}	\label{fig:delta3}
    \end{figure}
	
	It can be seen that $\delta_2$ either decreases from $-2 \mu _h^3 \lambda _l-2 \mu _h^2 \lambda _l^2+\mu _h^2 \lambda _l \mu _l$ and then increases to $\infty$, or increases to $\infty$. In any case, there is a unique value of $\lambda_h$ on $(0, \infty)$ that makes $\delta_2 = 0$.
	
   For any fixed $\lambda_l,\mu_h,\mu_l$, when $\lambda_h = \lambda_l\,\mu_h/\mu_l$,
	\begin{equation*}
		\delta_1 = \frac{\mu _h^2 \lambda _l^2 \left(\mu _l-\mu _h\right)}{\mu _l} < 0\,, \qquad\qquad \delta_2 = \mu _h^2 \lambda _l \left(\mu _l-\mu _h\right) < 0 \,.
	\end{equation*}
	This is the case only if $\delta_1$ and $\delta_2$ intersect when they are negative, as depicted in Figure \ref{fig:delta3}. Thus, if we alter the value of $\lambda_h$, for any fixed $\lambda_l,\mu_h,\mu_l$,
	\begin{itemize}
		\item when $S(11|0) > S({*}1|1)$, $S(10|0) > S(11|0)$,
		\item when $S(11|0) > S(10|0)$, $S({*}1|1) > S(11|0)$. 
        \end{itemize}

\section{Proof of Theorem \ref{thm:n2}} \label{appendix:thmoptimal}

	We first calculate
	\begin{small}
	\begin{equation*}
	    S(0{*}|0) - S(00|1) = \frac{2 \, \lambda _l^2 \,  \delta_1}{\left(2 \lambda _l \mu _l+2 \lambda _l^2+\mu _l^2\right) \left(2 \lambda _l^2 \left(\lambda _h+\mu _h\right) \left(\lambda _h+\lambda _l\right)+2 \lambda _l \mu _l \left(\mu _h \lambda _l+\lambda _h \left(\mu _h+\lambda _l\right)\right)+\lambda _h \mu _h \mu _l^2\right)} \,,
	\end{equation*}
	\end{small}
	where
	\begin{small}
	\begin{equation*}
		\delta_1 = 2 \lambda _h^2 \mu _l \left(\lambda _l+\mu _l\right)+\lambda _h \left(\mu _l^2 \left(4 \lambda _l-\mu _h\right)+2 \lambda _l \mu _l \left(\lambda _l-\mu _h\right)-2 \mu _h \lambda _l^2+2 \mu _l^3\right)-\mu _h \lambda _l \left(4 \lambda _l \mu _l+2 \lambda _l^2+\mu _l^2\right)\,.
	\end{equation*}
	\end{small}
	For any fixed value of $\lambda_l$ and $\mu_l$, there is a unique positive value
	\begin{align}
		\tilde{\lambda}_h \equiv  f(\mu_h) =   \frac{\mu _h \, A + B + \sqrt{\mu^2_h \, C + \mu_h \, D + E}}{4 \mu _l \left(\lambda _l+\mu _l\right)} \notag \,,
	\end{align}
	where
	\begin{align*}
	    & A = 2 \lambda _l \mu _l+2 \lambda _l^2+\mu _l^2  \\
	    & B = -4 \lambda _l \mu _l^2-2 \lambda _l^2 \mu _l-2 \mu _l^3 \\
	    & C = 8 \lambda _l^3 \mu _l+8 \lambda _l^2 \mu _l^2+4 \lambda _l \mu _l^3+4 \lambda _l^4+\mu _l^4 \\
	    & D = -16 \lambda _l \mu _l^4-28 \lambda _l^2 \mu _l^3-24 \lambda _l^3 \mu _l^2-8 \lambda _l^4 \mu _l+8 \lambda _l \mu _l \left(\lambda _l+\mu _l\right) \left(4 \lambda _l \mu _l+2 \lambda _l^2+\mu _l^2\right)-4 \mu _l^5 \\
	    & E = 16 \lambda _l \mu _l^5+24 \lambda _l^2 \mu _l^4+16 \lambda _l^3 \mu _l^3+4 \lambda _l^4 \mu _l^2+4 \mu _l^6 \,.
	\end{align*}
    Also, $\delta_1 = -\mu _h \lambda _l \left(4 \lambda _l \mu _l+2 \lambda _l^2+\mu _l^2\right) < 0$ when $\lambda_h = 0$, and the parameter in front of $\lambda_h^2$ is $2 \mu _l \left(\lambda _l+\mu _l\right) > 0$. Thus $0{*}|{0}$ is better than $00|1$ when $\lambda_h < f(\mu_h)$, and $00|1$ is better than $0{*}|{0}$ when $\lambda_h > f(\mu_h)$.
	As $\mu_h \rightarrow \infty$,
	\[
	\lim_{\mu_h \rightarrow \infty}\frac{f(\mu_h)}{\mu_h}
	\approx \frac{A + \sqrt{C}}{4 \mu _l \left(\lambda _l+\mu _l\right)} \,.
	\]
	That is, when $\lambda_l$ and $\mu_l$ are fixed, and $\mu_h$ is very large, the switching curve that separates the regions where $0{*}|0$ is better and $00|1$ is better, is approximately linear.
	
	Next, we calculate
	{\footnotesize \begin{align*}
		S({*}1|1) - S(00|1) = \frac{2 \mu _h \delta_2}{\left(2 \lambda _h \mu _h+2 \lambda _h^2+\mu _h^2\right) \left(2 \lambda _l^2 \left(\lambda _h+\mu _h\right) \left(\lambda _h+\lambda _l\right)+2 \lambda _l \mu _l \left(\mu _h \lambda _l+\lambda _h \left(\mu _h+\lambda _l\right)\right)+\lambda _h \mu _h \mu _l^2\right)} \,,
	\end{align*}
	}where
	\begin{align}
		\delta_2 = -\lambda _h^3 \left(2 \lambda _l \mu _l+2 \mu _l^2\right)+\lambda _h^2 \left(2 \mu _h \lambda _l^2-\mu _h \mu _l^2-2 \lambda _l^2 \mu _l\right)+\lambda _h \left(2 \mu _h \lambda _l^3+\mu _h^2 \lambda _l^2+\mu _h^2 \lambda _l \mu _l\right)+\mu _h^2 \lambda _l^3\,. \notag
	\end{align}
	We have
	\begin{align}
		\frac{\partial \delta_2}{\partial \lambda_h} = -\lambda _h^2 \left(6 \lambda _l \mu _l+6 \mu _l^2\right) +\lambda _h \left(4 \mu _h \lambda _l^2-2 \mu _h \mu _l^2-4 \lambda _l^2 \mu _l\right)  + 2 \mu _h \lambda _l^3+\mu _h^2 \lambda _l^2+\mu _h^2 \lambda _l \mu _l \,. \notag
	\end{align}
	It is apparent that $\frac{\partial \delta_2}{\partial \lambda_h}$ is positive first and then becomes negative. Thus, $\delta_2$ is increasing in $\lambda_h$ and then decreasing to $-\infty$. So there is a unique value $\tilde{\lambda}_h \equiv g(\mu_h)$ such that $S(00|1) \leq S({*}1|1)$ when $0 < \lambda_h \leq \tilde{\lambda}_h$, and $S(00|1) > S({*}1|1)$ when $\lambda_h > \tilde{\lambda}_h$. The expression of $g(\mu_h)$ is too complicated, we omit it here, but provide the limiting slope.
	
	Similarly, when $\mu_h$ is large enough,
	{\scriptsize
	\begin{align*}
	    & \lim_{\mu_h \rightarrow \infty} \frac{g(\mu_h)}{\mu_h} = \frac{1}{12 \mu _l \left(\lambda _l+\mu _l\right)} \left( 4 \lambda _l^2-2 \mu _l^2 \right. \\
		& + (1+i \sqrt{3}) \sqrt[3]{-18 \lambda _l^5 \mu _l-24 \lambda _l^4 \mu _l^2-9 \lambda _l^3 \mu _l^3+12 \lambda _l^2 \mu _l^4+9 \lambda _l \mu _l^5-i \, 3 \sqrt{3} \sqrt{\lambda _l^2 \mu _l^2 \left(\lambda _l+\mu _l\right){}^4 \left(8 \lambda _l^3 \mu _l+12 \lambda _l^2 \mu _l^2+8 \lambda _l \mu _l^3+4 \lambda _l^4+\mu _l^4\right)}-8 \lambda _l^6+\mu _l^6}  \\
		& \left. + (1-i \sqrt{3}) \sqrt[3]{-18 \lambda _l^5 \mu _l-24 \lambda _l^4 \mu _l^2-9 \lambda _l^3 \mu _l^3+12 \lambda _l^2 \mu _l^4+9 \lambda _l \mu _l^5+i \, 3 \sqrt{3} \sqrt{\lambda _l^2 \mu _l^2 \left(\lambda _l+\mu _l\right){}^4 \left(8 \lambda _l^3 \mu _l+12 \lambda _l^2 \mu _l^2+8 \lambda _l \mu _l^3+4 \lambda _l^4+\mu _l^4\right)}-8 \lambda _l^6+\mu _l^6}  \right) \,.
	\end{align*}}That is, when $\lambda_l$ and $\mu_l$ are fixed, and $\mu_h$ is very large, the switching curve the separates the regions where $00|1$ is better and $*1|1$ is better, is approximately linear.
	
	\begin{figure}[h]
    \centering
	\includegraphics[width=0.6\linewidth]{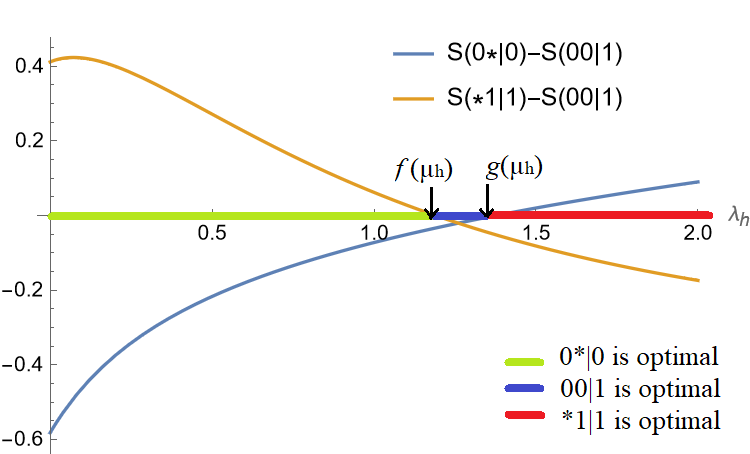}
	\caption{$\mu _h = 1.8,\,\lambda_l= 0.7$.}	\label{fig:delta4}
    \end{figure}

    From the above analysis, we know that for $\lambda_h >0$, $S(0{*}|0) - S(00|1) > 0$ if and only if $\lambda_h<f(\mu_h)$, and $S(*1|1) - S(00|1) < 0$ if and only if $\lambda_h<g(\mu_h)$. To see that $f(\mu_h) < g(\mu_h)$, we first check
    {\footnotesize\begin{align*}
        & S(0{*}|0) - S(00|1) \mid_{\lambda_h = \frac{\lambda_l}{\mu_l}\,\mu_h} = \frac{2 \lambda _l^2 \mu _l^3 \left(\mu _l-\mu _h\right)}{\left(2 \lambda _l \mu _l+2 \lambda _l^2+\mu _l^2\right) \left(2 \lambda _l^3 \left(\mu _h+\mu _l\right)+2 \lambda _l^2 \mu _l \left(\mu _h+2 \mu _l\right)+2 \lambda _l \mu _l^2 \left(\mu _h+\mu _l\right)+\mu _h \mu _l^3\right)} < 0 \\
        & S(*1|1) - S(00|1) \mid_{\lambda_h = \frac{\lambda_l}{\mu_l}\,\mu_h} =\frac{2 \lambda _l^2 \mu _l^3 \left(\mu _l-\mu _h\right)}{\left(2 \lambda _l \mu _l+2 \lambda _l^2+\mu _l^2\right) \left(2 \lambda _l^3 \left(\mu _h+\mu _l\right)+2 \lambda _l^2 \mu _l \left(\mu _h+2 \mu _l\right)+2 \lambda _l \mu _l^2 \left(\mu _h+\mu _l\right)+\mu _h \mu _l^3\right)} < 0 \,.
    \end{align*}
    }Thus, $S(0{*}|0) - S(00|1)$ and $S(*1|1) - S(00|1)$ intersect when both of them are negative, as depicted in Figure \ref{fig:delta4}, so $\lambda_h<g(\mu_h)$, and we prove the results. The optimal strategy along $\lambda_h$ axis is also depicted in Figure \ref{fig:delta4}. Also note that when $\lambda_h/\mu_h = \lambda_l/\mu_l$, $00|1$ is optimal.

\section{Calculation of the system efficiency} \label{appendix:QS}

There are four types of transitions: a customer in high rate becomes inactive;  a customer in low rate becomes inactive; an inactive customer becomes active in low rate; an inactive customer becomes active in high rate. The transition rates are listed in Table \ref{tab:TRates2}.

\begin{table}[!ht]
\caption{Transition rates when every customer adopts threshold $n$.} \label{tab:TRates2}
\vspace*{10pt}
\centering
        {\def\arraystretch{1.5} 
\begin{tabular*}{0.5\textwidth}{@{\extracolsep{\fill}}lcl}
\hline
\hline
Transition  & Rate & States\\
		\hline \hline
		$(i,h) \rightarrow (i+1, h-1)$ & $h \lambda_h$ & $h > 0$\\
		\hline
		$(i,h) \rightarrow (i+1,h)$ & $(N-i-h) \lambda_l$ & $i+h < N$\\
		\hline
		$(i,h) \rightarrow (i-1,h)$ & $\mu_l$ & $i \leq n$ \\
		\hline
		$(i,h) \rightarrow (i-1,h+1)$ & $\mu_h$ & $i > n$ \\
\hline
\hline
\end{tabular*}
\label{table:mechanical_property}
}
\end{table}

By treating $i$ as the phase and $h$ as the level, we can calculate the stationary distribution of the states in $\mathcal{S}$ via matrix analytic methods \citep{N81}. First, we organize the states as
\begin{align*}
	(0,0),(1,0), \ldots, (N,0), (0,1), (1,1), \ldots, (N-1,1), \ldots, (0,N) \,.
\end{align*}
We then write down the generator matrix under threshold strategy $n$ as
\begin{equation*}
	Q_S(n) =
	\begin{bmatrix}
		Q_0^{(0)} & Q_1^{(0)} & 0 & 0  &  \\
		Q_{-1}^{(1)} & Q_0^{(1)} & Q_1^{(1)} & 0  &   \\
		0 & Q_{-1}^{(2)} & Q_0^{(2)} & Q_1^{(2)} &  \\
		&  & \ddots & \ddots  &  \\
		&  &  &  Q_{-1}^{(N)}  & Q_0^{(N)}
	\end{bmatrix} \,,
\end{equation*}
where the defining matrices $Q_{-1}^{(h)}, Q_{0}^{(h)}, Q_{1}^{(h)}$ are
\begin{align*}
	& Q^{(h)} _{-1}  =
	\begin{bmatrix}
		0  & h \lambda_h &  &    \\
		& 0  & h \lambda_h  &  & \\
		&  & \ddots & \ddots   &    \\
		&  &   & 0& h \lambda_h  \\
	\end{bmatrix} \qquad\qquad
    Q_1^{(h)} =
	\begin{bmatrix}
		0  &  &  &    \\
		\lambda_h & 0  &    &  && \\
		&  \lambda_h & 0 &    &  &  \\
		&   & \ddots & \ddots   & &   \\
		&  &   & 0& \ddots &  \\
		&  &   &  &\ddots & 0 \\
		&  &   &  & & 0
	\end{bmatrix}  \\
	& Q^{(h)} _{-1} \in \mathbb{R}^{(N-h+1) \times(N-h+2)}, \, Q_1^{(h)}  \in  \mathbb{R}^{(N-h+1) \times(N-h)}\,,  \\
	& Q_0^{(h)} =
	\begin{bmatrix}
		b_0^{(h)} & (N-h) \lambda_l & 0 &  & &  &  \\
		0 & b_1^{(h)}  & (N-h-1) \lambda_l & &  &  &  \\
		& \ddots & \ddots  &  \ddots &  &   & \\
		&   & \lambda_l &  b_n^{(h)}  &  (N-h-n) \lambda_l  & & \\
		&  &   & \ddots & \ddots&   \\
		&  & & & \lambda_l & b_c^{(h)} &
	\end{bmatrix}\in\mathbb{R}^{(N-h+1) \times(N-h+1)} \,,
\end{align*}
$b_j^{(k)}$ is the number that makes the row sum of $Q$ equal to $0$.

The stationary distribution vector
\[
\bm{\pi}^{(n)} \equiv \left[\pi^{(n)}_{0,0},\pi^{(n)}_{1,0}, \ldots, \pi^{(n)}_{N,0}, \pi^{(n)}_{0,1}, \pi^{(n)}_{1,1}, \ldots, \pi^{(n)}_{N-1,1}, \ldots, \pi^{(n)}_{0,N}\right] \,,
\]
can be obtained by solving
\[
\bm{\pi}^{(n)} Q_S(n) = 0 \qquad  \bm{\pi}^{(n)}\bm{e} = 1\,,
\]
where $\bm{e}$ is a column vector of ones. With $\bm{\pi}^{(n)}$, the efficiency, when everyone is using threshold $n$, can be calculated using formula \eqref{eq:SW} or \eqref{eq:SW2}.

\section{Calculation of the individual efficiency} \label{appendix:QN}

Note that when $h = 0$, there are no active customers in high rate, thus $k$ cannot take value $0$. When $i+h = N$, there are no active customers at a low rate, so $k$ cannot take value $-1$. In the analysis of the best-response threshold strategy, we assume that the tagged customer uses threshold $m \in \{0,1,\ldots, N\}$, and other customers use threshold $x \in [0,N]$. To calculate the stationary distribution of $\pi^{(m,x)}_{k,i,h}$, we first write down the state space
{\footnotesize \begin{align} \label{eq:state}
	\mathcal{S}' = \{&(-1,0,0),\,(-1,1,0), \,(1,1,0 ), \, (-1,2,0),\,(1,2,0 ),\,(2,2,0 ), \ldots, (1,N,0), \ldots, (N,N,0) \notag \\
	& \ldots \notag \\
	& (-1,0,h),\, (0,0,h), \, (-1,1,h), \,(0,1,h), \,(1,1,h), \ldots, (0,N-h,h),\ldots, (N-h,N-h,h) \notag \\
	& \ldots \notag \\
	& (0,0,N) \} \,.
	\end{align}
	}

    When the event is an active customer at a high rate becomes inactive, $i$ increases by $1$, and $h$ decreases by $1$. If $k=-1$ or $k\geq 1$, then the one who becomes inactive must be from customers other than the tagged customer. Thus with rate $h \lambda_h$, an active customer with a high rate becomes inactive. If $k = 0$, the tagged customer is active with a high rate. If it is not the tagged customer that becomes inactive, then the rate is $(h-1) \, \lambda_h$, and $k$ stays at $0$. If it is the tagged customer becomes inactive, the rate is $\lambda_h$, and $k$ becomes $i+1$.

    When the event is an active customer at a low rate becomes inactive, $i$ increases by $1$, and $h$ stays the same. If $k = 0$ or $k \geq 1$, the tagged customer is active at a high rate, so the one who becomes inactive must be from customers other than the tagged customer. With rate $(N-i-h) \lambda_l$, an active customer with a low rate becomes inactive. If $k = -1$, the tagged customer is active with a low rate. If it is not the tagged customer who becomes inactive, then the rate is $(N-i-h-1) \, \lambda_l$, and $k$ stays at $-1$. If it is the tagged customer becomes inactive, the rate is $\lambda_l$, and $k$ becomes $i+1$.

    When the event is a customer finishes her service at a low rate, $i$ decreases by $1$, and $h$ stays the same. If $k = 0,-1$ or $k>1$, the tagged customer is either queueing or active, so the one who departs is not the tagged customer. The customer in service departs with rate $\mu_l$ if $i>n+1$ and rate $\mu_l \, (1-p)$ if $i=n+1$. State $k$ decreases by $1$ if $k>1$ and stays the same if $k = 0,-1$. If $k=1$, it can only be $i > m$, the tagged customer leaves with rate $\mu_l$.

    When the event is a customer finishes her service in high rate, $i$ decreases by $1$, and $h$ increases by $1$. If $k = 0,-1$ or $k>1$, the tagged customer is either queueing or active, so the one who departs is not the tagged customer. The customer in service departs with rate $\mu_h$ if $i\leq n$ and rate $\mu_h \, p$ if $i=n+1$. State $k$ decreases by $1$ if $k \geq 1$ and stays the same if $k = 0,-1$. If $k=1$, it must be $i\leq m$, and the tagged customer leaves with rate $\mu_h$.

\begin{table}[!ht]
\caption{Transition rates when the tagged customer uses $m$ and all the others use $x$.} \label{tab:TRates}
\vspace*{10pt}
\centering
        {\def\arraystretch{1.5} 
\begin{tabular*}{0.8\textwidth}{@{\extracolsep{\fill}}lcl}
\hline
\hline
	Transition  & Rate & States\\
		\hline \hline
		\Xhline{3\arrayrulewidth}
		$(k,i,h) \rightarrow (k, i+1, h-1)$ & $h \lambda_h$ & $h > 0, \, k \geq 1$ or $ k = -1$ \\
		\hline
		$(0,i,h) \rightarrow (0, i+1, h-1)$ & $(h-1) \lambda_h$ & $h>0$ \\
		\hline
		$(0,i,h) \rightarrow (i+1, i+1, h-1)$ & $\lambda_h$ & $h>0$ \\
		\Xhline{3\arrayrulewidth}
		$(k,i,h) \rightarrow (k,i+1,h)$ & $(N-i-h) \lambda_l$ & $i+h < N$, $k \geq 1$ or $k = 0$ \\
		\hline
		$(-1,i,h) \rightarrow (-1,i+1,h)$ & $(N-i-h-1) \lambda_l$ & $i+h<N$ \\
		\hline
		$(-1,i,h) \rightarrow (i+1,i+1,h)$ & $ \lambda_l$ & $i+h<N$ \\
		\Xhline{3\arrayrulewidth}
		$(k,i,h) \rightarrow (k-1,i-1,h)$ & $\mu_l \, (1-p)$ & $k > 1, \, i = n+1$ \\
		\hline
		$(k,i,h) \rightarrow (k-1,i-1,h)$ & $\mu_l$ & $k > 1, \, i > n+1$ \\
		\hline
		$(k,i,h) \rightarrow (k,i-1,h)$ & $\mu_l \, (1-p)$ & $k = 0, -1, \, i = n+1$ \\
		\hline
		$(k,i,h) \rightarrow (k,i-1,h)$ & $\mu_l$ & $k = 0, -1, \, i > n+1$ \\
		\hline
		$(1,i,h) \rightarrow (-1,i-1,h)$ & $\mu_l$ & $i > m$ \\
		\Xhline{3\arrayrulewidth}
		$(k,i,h) \rightarrow (k-1,i-1,h+1)$ & $\mu_h \, p$ & $k>1, i = n+1$ \\
		\hline
		$(k,i,h) \rightarrow (k-1,i-1,h+1)$ & $\mu_h$ & $k>1,i \leq n$ \\
		\hline
		$(k,i,h) \rightarrow (k,i-1,h+1)$ & $\mu_h \, p $ & $k = 0, -1, \, i = n+1$ \\
		\hline
		$(k,i,h) \rightarrow (k,i-1,h+1)$ & $\mu_h$ & $k = 0, -1, \, i \leq n$ \\
		\hline
		$(1,i,h) \rightarrow (0,i-1,h+1)$ & $\mu_h$ & $i \leq m$ \\
\hline
\hline
\end{tabular*}
\label{table:mechanical_property}
}
\end{table}

The transition rates described above are summarized in Table \ref{tab:TRates}. If we line up the states in the order shown in \eqref{eq:state}, the transition rate matrix can be written as
\[
Q_N(x) =
\begin{bmatrix}
B_{0}^{(0)} & B_{1}^{(0)}& \ldots & \ldots & \ldots  & \vdots \\
B_{-1}^{(1)} & B_{0}^{(1)}  & B_{1}^{(1)}  & \ldots & \ldots  & \vdots \\
\bm{0}& B_{-1}^{(2)} & B_{0}^{(2)}  & B_{1}^{(2)}  & \ldots   & \vdots \\
\vdots & \ddots & \ddots & \ddots &&\vdots \\
\vdots&\ldots&\ldots&\ldots& B_{-1}^{(N)} & B_{0}^{(N)}
\end{bmatrix}
\]
has the block partition form, where the defining matrices $B_{-1}^{(h)}$, $B_{0}^{(h)},B_{1}^{(h)}$ are
\begin{align*}
& B_{-1}^{(h)} = \begin{bmatrix}
	\bm{0} & A_{{-1,1}}^{(0,h)} & \ldots & \ldots  & \ldots & \vdots\\
	\bm{0} & \bm{0}  & A_{{-1,1}}^{(1,h)} & \ldots & \ldots & \vdots\\
	\bm{0} & \bm{0} & \bm{0} & A_{{-1,1}}^{(2,h)}  & \ldots &\vdots\\
	\vdots & \vdots & \vdots & \ddots & \ddots &\vdots \\
	\vdots&\ldots&\ldots& \ldots& \bm{0} & A_{{-1,1}}^{(N-h,h)}
\end{bmatrix}  \qquad h = 1, 2, \ldots, N \,,\\
& B_{0}^{(h)} =
\begin{bmatrix}
A_{0,0}^{(0,h)} & A_{{0,1}}^{(0,h)} & \ldots & \ldots   & \vdots \\
A_{0,-1}^{(1,h)} & A_{0,0}^{(1,h)} & A_{{0,1}}^{(1,h)}  & \ldots  & \vdots \\
\bm{0} & A_{0,-1}^{(2,h)} & A_{0,0}^{(2,h)} & \ldots   & \vdots \\
\vdots & \ddots & \ddots  &&\vdots \\
\vdots&\ldots&\ldots& A_{0,-1}^{(N-h,h)} & A_{0,0}^{(N-h,h)}
\end{bmatrix} \qquad h = 0, 1, \ldots, N, \\
& B_{1}^{(h)} =
\begin{bmatrix}
\bm{0} & \ldots & \ldots & \ldots & \vdots  \\
A_{{1,-1}}^{(1,h)}  & \bm{0}  & \ldots & \ldots & \vdots  \\
\bm{0} & A_{{1,-1}}^{(2,h)}  & \bm{0} & \ldots & \vdots  \\
\vdots & \vdots & \ddots & \ddots & \vdots \\
\vdots&\ldots&\ldots& A_{{1,- 1}}^{(N-h,h)} & \bm{0} &
\end{bmatrix} \qquad h = 0, 1, \ldots, N-1 \,,
\end{align*}
where $A_{0,1}^{(i,h)}, A_{0,-1}^{(i,h)}, A_{-1,1}^{(i,h)}$ and $A_{1,-1}^{(i,h)}$ are defined as follows. Let $I_n$ denote the identity matrix with dimension $n \times n$. The matrix form for $A_{0,1}^{(i,h)}$ is
\begin{itemize}
    \item when $i = 0,1,\ldots, N-2-h$,
    \begin{equation*}
        A_{0,1}^{(i,h)} =
        \begin{bmatrix}
            (N-1-i) \, \lambda_l & & \lambda_l \\
             & (N-i)\lambda_l \, I_i
        \end{bmatrix} \,,
    \end{equation*}
    \item when $i = N-1-h$,
    \begin{equation*}
        A_{0,1}^{(N-1-h,h)} =
        \begin{bmatrix}
             & \lambda_l \\
             \lambda_l \, I_i &
        \end{bmatrix} \,.
    \end{equation*}
\end{itemize}

The matrix form for $A_{0,-1}^{(i,h)}$ is
\begin{itemize}
    \item when $h = 0$, if $i < N$,
    \begin{equation*}
        A_{0,-1}^{(i,0)} =
        \begin{bmatrix}
            \mu_l(1-p)\mathbbm{1}(i = n+1) + \mu_l\mathbbm{1}(i > n+1) & \\
            \mu_l\mathbbm{1}(i > m) & \\
             & \mu_l(1-p)\mathbbm{1}(i = n+1) + \mu_l\mathbbm{1}(i > n+1) \\
             & I_{i-1}
        \end{bmatrix} \,,
    \end{equation*}
    if $i = N$
    \begin{equation*}
        A_{0,-1}^{(N,0)} =
        \begin{bmatrix}
            \mu_l\mathbbm{1}(i > m) & \\
             & (\mu_l(1-p)\mathbbm{1}(i = n+1) + \mu_l\mathbbm{1}(i > n+1)) \,  I_{i-1}
        \end{bmatrix} \,,
    \end{equation*}
    \item when $h > 0$, if $i < N-h$,
    {\footnotesize\begin{equation*}
        A_{0,-1}^{(i,h)} =
        \begin{bmatrix}
            (\mu_l(1-p)\mathbbm{1}(i = n+1) + \mu_l\mathbbm{1}(i > n+1)) \, I_2 & \\
            \mu_l\mathbbm{1}(i > m) \qquad 0 &  \\
             & (\mu_l(1-p)\mathbbm{1}(i = n+1) + \mu_l\mathbbm{1}(i > n+1)) \, I_{i-1}
        \end{bmatrix} \,,
    \end{equation*}}
    if $i = N-h$,
    {\footnotesize \begin{equation*}
        A_{0,-1}^{(N-h,h)} =
        \begin{bmatrix}
             & \mu_l(1-p)\mathbbm{1}(i = n+1) + \mu_l\mathbbm{1}(i > n+1) & \\
            \mu_l\mathbbm{1}(i > m)  & & \\
             & & (\mu_l(1-p)\mathbbm{1}(i = n+1) + \mu_l\mathbbm{1}(i > n+1)) \, I_{i-1}
        \end{bmatrix} \,.
    \end{equation*}}
\end{itemize}
The matrix form for $A_{-1,1}^{(i,h)}$ is
\begin{itemize}
    \item when $h = 1$, if $i < N-1$,
    \begin{equation*}
    A_{-1,1}^{(i,1)} =
    \begin{bmatrix}
        \lambda_h & & \\
        & & \lambda_h \\
        & \lambda_h \, I_i &
    \end{bmatrix} \,,
    \end{equation*}
    if $i = N-1$,
    \begin{equation*}
    A_{-1,1}^{(N-1,1)} =
    \begin{bmatrix}
         & \lambda_h \\
        \lambda_h \, I_{N-1} &
    \end{bmatrix} \,,
    \end{equation*}
    \item when $h>1$, if $i < N-h$,
    \begin{equation*}
    A_{-1,1}^{(i,h)} =
    \begin{bmatrix}
        h \lambda_h & & & \\
        & (h-1) \lambda_h & & \lambda_h\\
        & & \lambda_h \, I_i  &
    \end{bmatrix} \,,
    \end{equation*}
    if $i = N-h$,
    \begin{equation*}
    A_{-1,1}^{(N-h,h)} =
    \begin{bmatrix}
        (h-1) \lambda_h &  & \lambda_h \\
        & h \lambda_h \, I_i  &
    \end{bmatrix} \,.
    \end{equation*}
\end{itemize}
The matrix form for $A_{1,-1}^{(i,h)}$ is
\begin{itemize}
    \item when $h = 0$, if $i < N$,
    \begin{equation*}
        A_{1,-1}^{(i,0)} =
        \begin{bmatrix}
            \mu_h \, p \, \mathbbm{1}(i = n+1) + \mu_h \, \mathbbm{1}(i \leq n)  & & \\
            & \mu_h \, \mathbbm{1}(i \leq m) & \\
            & & (\mu_h \, p \, \mathbbm{1}(i = n+1) + \mu_h \, \mathbbm{1}(i \leq n)) I_{i-1}
        \end{bmatrix}
    \end{equation*}
    if $i = N$,
    \begin{equation*}
        A_{1,-1}^{(N,0)} =
        \begin{bmatrix}
            \mu_h \, \mathbbm{1}(i \leq m) & \\
            & (\mu_h \, p \, \mathbbm{1}(i = n+1) + \mu_h \, \mathbbm{1}(i \leq n)) I_{N-1}
        \end{bmatrix} \,,
    \end{equation*}
    when $h > 0$, if $i < N-h$,
    \begin{equation*}
        A_{1,-1}^{(i,h)} =
        \begin{bmatrix}
            (\mu_h \, p \, \mathbbm{1}(i = n+1) + \mu_h \, \mathbbm{1}(i \leq n)) \, I_2  & \\
            0 \qquad\qquad \mu_h \, p \, \mathbbm{1}(i \leq m) & \\
            & (\mu_h \, p \, \mathbbm{1}(i = n+1) + \mu_h \, \mathbbm{1}(i \leq n)) I_{i-1}
        \end{bmatrix} \,,
    \end{equation*}
    if $i = N-h$,
    \begin{equation*}
        A_{1,-1}^{(N-h,h)} =
        \begin{bmatrix}
            \mu_h \, p \, \mathbbm{1}(i = n+1) + \mu_h \,\mathbbm{1}(i \leq n) & \\
            \mu_h \, \mathbbm{1}(i \leq m) & \\
            & (\mu_h \, p \, \mathbbm{1}(i = n+1) + \mu_h \, \mathbbm{1}(i \leq n)) I_{i-1}
        \end{bmatrix} \,.
    \end{equation*}
\end{itemize}

Then the stationary distribution vector
\[
\bm{\pi}^{(m,x)} \equiv \left[\pi^{(m,x)}_{-1,0,0},\pi^{(m,x)}_{-1,1,0}, \ldots, \pi^{(m,x)}_{0,0,N}\right] \,,
\]
can be obtained by solving
\[
\bm{\pi}^{(m,x)} Q_N(x) = 0 \qquad  \bm{\pi}^{(m,x)}\bm{e} = 1\,.
\]

\end{document}